\pdfoutput=1
\documentclass[11pt, a4paper]{article}
\usepackage[top=1.5545in,bottom=1.5545in,left=1.5545in, right=1.5545in]{geometry}
\usepackage[svgnames]{xcolor}
\usepackage{mathtools}
\usepackage[hyphens]{url}
\usepackage[abbrev,lite,nobysame]{amsrefs}
\usepackage{amsthm,amsfonts,amsmath}
\usepackage{abstract}
\usepackage{hyphenat}
\usepackage{csquotes}
\usepackage{enumerate}
\usepackage[font=footnotesize,labelfont=bf]{caption}
\usepackage{subcaption}
\usepackage[colorlinks=true,linkcolor=NavyBlue,
citecolor=NavyBlue,urlcolor=NavyBlue]{hyperref}
\usepackage{cleveref}
\crefname{equation}{}{}

\usepackage{tocloft}
\usepackage[normalem]{ulem}

\usepackage[looser]{newtxtext}
\usepackage[bigdelims,varvw]{newtxmath}
\usepackage[protrusion=true, tracking=true, kerning=true,spacing=true]{microtype}
\usepackage{bm}

\usepackage{tikz}
\usepgflibrary{patterns}
\usetikzlibrary{patterns.meta,decorations.pathmorphing
	,decorations.text,positioning,shapes.geometric,arrows.meta,decorations.markings,fadings,decorations.pathreplacing,calligraphy} 
\usepackage{pgfplots}
\pgfplotsset{compat = newest}
\usepgfplotslibrary{colormaps}
\tikzfading[name=fade right,
left color=transparent!0, right color=transparent!100]
\tikzfading[name=fade left,
right color=transparent!0, left color=transparent!100]
\tikzset{>={Latex[width=3mm,length=3mm]}}

\tikzstyle{B1} = [rectangle, rounded corners, minimum width=2.2cm, minimum height=1.4cm,text centered, draw=black, fill=MidnightBlue!3, text width=3.9cm]
\tikzstyle{B11} = [rectangle, rounded corners, minimum width=2.2cm, minimum height=1.4cm,text centered, draw=black, fill=MidnightBlue!3, text width=9cm]

\tikzstyle{B0} = [rectangle, rounded corners, minimum width=2.2cm, minimum height=1.4cm,text centered, draw=black, fill=MidnightBlue!3, text width=8.8cm]
\tikzstyle{B00} = [rectangle, rounded corners, minimum width=2.2cm, minimum height=1.4cm,text centered, draw=black, fill=MidnightBlue!3, text width=11.6cm]
\tikzstyle{B000} = [rectangle, rounded corners, minimum width=2.2cm, minimum height=1.4cm,text centered, draw=black, fill=MidnightBlue!3, text width=5cm]

\tikzstyle{blue0} = [rectangle, rounded corners, minimum width=2.2cm, minimum height=1.4cm,text centered, draw=black, fill=MidnightBlue!3, text width=2.1cm]
\tikzstyle{blue} = [rectangle, rounded corners, minimum width=3.6cm, minimum height=1.4cm,text centered, draw=black, fill=MidnightBlue!3, text width=4.5cm]
\tikzstyle{blue1} = [rectangle, rounded corners, minimum width=6cm, minimum height=1.4cm,text centered, draw=black, fill=MidnightBlue!3, text width=6.2cm]

\tikzstyle{arrow} = [thick,->]

\tikzdeclarepattern{
	name=mylines,
	parameters={
		\pgfkeysvalueof{/pgf/pattern keys/size},
		\pgfkeysvalueof{/pgf/pattern keys/angle},
		\pgfkeysvalueof{/pgf/pattern keys/line width},
	},
	bounding box={
		(0,-0.5*\pgfkeysvalueof{/pgf/pattern keys/line width}) and
		(\pgfkeysvalueof{/pgf/pattern keys/size},
		0.5*\pgfkeysvalueof{/pgf/pattern keys/line width})},
	tile size={(\pgfkeysvalueof{/pgf/pattern keys/size},
		\pgfkeysvalueof{/pgf/pattern keys/size})},
	tile transformation={rotate=\pgfkeysvalueof{/pgf/pattern keys/angle}},
	defaults={
		size/.initial=5pt,
		angle/.initial=45,
		line width/.initial=.4pt,
	},
	code={
		\draw [line width=\pgfkeysvalueof{/pgf/pattern keys/line width}]
		(0,0) -- (\pgfkeysvalueof{/pgf/pattern keys/size},0);
	},
}

\numberwithin{equation}{section}
\theoremstyle{plain}
\newtheorem{thrm}{Theorem}[section]

\newtheorem{lmm}[thrm]{Lemma}
\newtheorem{crllr}[thrm]{Corollary}
\newtheorem{prpstn}[thrm]{Proposition}

\theoremstyle{definition}
\newtheorem{dfntn}[thrm]{Definition}

\newtheorem{rmrk}[thrm]{Remark}

\theoremstyle{plain}

\newcommand{\xvec}[1]{\bm{#1}}
\newcommand{\xsym}[1]{\bm{#1}}
\newcommand{\xdop}[1]{\bm{\mathrm{#1}}}
\def\xnab{\xdop{\nabla}}

\newcommand{\xwcurl}[1]{\xdop{\nabla}\wedge{#1}}

\newcommand{\xdiv}[1]{\xdop{\nabla}\cdot{#1}}
\newcommand{\xmcal}[1]{\bm{\mathcal{#1}}}
\newcommand{\xdx}[1]{{{\rm d}#1}}
\def\xdrv#1#2{\frac{{\rm d}#1}{{\rm d}#2}}
\def\cf{\emph{cf.\/}}\def\eg{\emph{e.g.\/}}\def\Eg{\emph{E.g.\/}}

\def\Apriori{\emph{A priori\/}}

\def\@Rref#1{\hbox{\rm \ref{#1}}}
\def\Rref#1{\@Rref{#1}}
\def\xCzero{{\rm C}^{0}}
\def\xCone{{\rm C}^{1}} 
 
\def\xCinfty{{\rm C}^{\infty}} 
\def\xCn#1{{\rm C}^#1}
\def\xH{{\rm H}}

\def\xHn#1{{\rm H}^#1}
\def\bHn#1{\mathbf{H}^#1}
\def\bVn#1{\mathbf{V}^#1}

\def\xWn#1{{\rm W}^#1}

\def\xLone{{\rm L}^{1}}
\def\xLtwo{{\rm L}^{2}}

\def\xLn#1{{\rm L}^#1}

\def\xB{{\rm B}}

\def\xI{{\rm I}}

\makeatletter
\newcommand*{\toccontents}{\@starttoc{toc}}
\makeatother

\begin{document}\linespread{1.05}\selectfont
	\date{}
	
	\author{Vahagn~Nersesyan\,\protect\footnote{NYU-ECNU Institute of Mathematical Sciences at NYU Shanghai, 3663 Zhongshan Road North, Shanghai, 200062, China, e-mail: \href{mailto:vahagn.nersesyan@nyu.edu}{Vahagn.Nersesyan@nyu.edu}}\and
		Manuel~Rissel\,\protect\footnote{NYU-ECNU Institute of Mathematical Sciences at NYU Shanghai, 3663 Zhongshan Road North, Shanghai, 200062, China, e-mail: \href{mailto:Manuel.Rissel@nyu.edu}{Manuel.Rissel@nyu.edu}}}
	
	\title{Global controllability of Boussinesq flows by using only a temperature control}

	\maketitle
	
	\begin{abstract}
		We show that buoyancy driven flows can be steered in an arbitrary time towards any state by applying as control only an external temperature profile in a subset of small area. More specifically, we prove that the $2$D incompressible Boussinesq system on the torus is globally approximately controllable via physically localized heating or cooling. In addition, our controls have an explicitly prescribed structure; even without such structural requirements, large data controllability results for Boussinesq flows driven merely by a physically localized temperature profile were so far unknown. The presented method exploits various connections between the model's underlying transport-, coupling-, and scaling mechanisms.
		
		\quad
		
		\begin{center}
		\end{center}
	
		\quad

	\end{abstract}

	\section{Introduction}
	We demonstrate the global approximate controllability of incompressible buoyancy-driven flows regulated only by a physically localized temperature control; \enquote{global} means in this context that the initial and target profiles might be far apart from each other in the state space. The considered incompressible Boussinesq system is relevant to the study of, \eg, geophysical phenomena and Rayleigh-B\'enard convection, and it also serves applications involving heating and ventilation (\cf~\cite{AbergelTemam1990,Getling1998,Tritton1977,FoiasManleyTemam1987}). In particular, it is desirable to uncover coupling mechanisms that facilitate the controllability of nonlinear fluids merely via regionally applied heating/cooling, without imposing smallness constraints on the data. This can matter from practical perspectives and provides deeper theoretical insights regarding the mathematical model itself.~But despite the variety of motivations, all available controllability results for the Boussinesq equations steered only by a temperature control have been limited to small perturbations of linear dynamics in~$2$D; see \Cref{subsection:literaturecontr} for bibliographical remarks. In this article, contrasting the existing literature, we tackle~a truly global controllability problem for the nonlinear Boussinesq system. The present approach even allows to fix the control's structure in terms of a small finite number of universal physically localized profiles. Our proof features geometric arguments,~a multi-stage scaling procedure, and the notion of transported Fourier modes from~\cite{NersesyanRissel2024}.
	
	\subsection{The main controllability problem}\label{subsection:maincontrollabilityproblem}
	Let $T > 0$ and assume that the gravitational field is given by $\xvec{e}_{\operatorname{grav}} \coloneq (0,1)$. The state of a viscous incompressible fluid in $\mathbb{T}^2 \coloneq \mathbb{R}^2 / 2\pi \mathbb{Z}^2$ is then described by means of its $2\pi$-periodic velocity, temperature, and exerted pressure; respectively,
	\[
		\xvec{u}\colon\mathbb{T}^2\times[0,T] \longrightarrow \mathbb{R}^2, \quad \theta\colon\mathbb{T}^2\times[0,T] \longrightarrow \mathbb{R}, \quad p \colon\mathbb{T}^2\times[0,T] \longrightarrow \mathbb{R},
	\]
	which are governed by the Boussinesq system
	\begin{equation}\label{equation:BoussinesqVelocity}
		\begin{gathered}
			\partial_t \xvec{u} - \nu \Delta \xvec{u} + \left(\xvec{u} \cdot \xnab\right) \xvec{u} + \xnab p = \theta \xvec{e}_{\operatorname{grav}} + \xsym{\Phi}_{\operatorname{ext}}, \quad \xdiv{\xvec{u}} = 0, \quad \xvec{u}(\cdot, 0) = \xvec{u}_0, \\ \partial_t \theta - \tau \Delta \theta + (\xvec{u}\cdot\xnab)\theta = \mathbb{I}_{\omegaup} \eta + \psi_{\operatorname{ext}}, \quad \theta(\cdot, 0) = \theta_0,
		\end{gathered}
	\end{equation}
	where~$\nu > 0$ is the viscosity and~$\tau > 0$ denotes the thermal diffusivity. Moreover, the functions~$\xsym{\Phi}_{\operatorname{ext}}$ and~$\psi_{\operatorname{ext}}$ represent given external forces, which are for simplicity assumed to have zero average. A distinguished role is played by the~unknown profile~$\eta$ in~\eqref{equation:BoussinesqVelocity}; it will act as the control, and its spatial support will be localized in an arbitrarily thin horizontal strip (\cf~\Cref{Figure:controldomain})
	\[
		\omegaup \coloneq \mathbb{T}\times (a, b), \quad 0 < a < b \leq 2\pi.
	\]
	
	\begin{figure}[ht!]
		\centering
		\resizebox{0.45\textwidth}{!}{
			\begin{tikzpicture}
				\clip(-0.05,-0.05) rectangle (9.05,9.05);

				\fill[line width=0.3mm, fill=CornflowerBlue!50] plot[smooth cycle] (0, 1.6) rectangle ++(9,2.5);

				\draw[line width=0.2mm, color=black, dashed, dash pattern=on 8pt off 8pt] plot[smooth cycle] (0,0) rectangle ++(9,9);
				
				\draw[line width = 0.8mm, ->] (2,5.5) -- (2, 8);
				
				\coordinate[label=right:\LARGE$\omegaup$] (A) at (4.1, 2.8);
				\coordinate[label=right:\LARGE$\xvec{e}_{\operatorname{grav}}$] (A) at (2.1, 6.75);
				
		\end{tikzpicture}}
		\caption{The control region $\omegaup \subset \mathbb{T}^2$ is any open horizontal strip and~$\xvec{e}_{\operatorname{grav}}$ points vertically.}
		\label{Figure:controldomain}
	\end{figure}
	
	Given any initial state $(\xvec{u}_0, \theta_0)$, target state $(\xvec{u}_T, \theta_T)$, control time~$T > 0$, and approximation error~$\varepsilon > 0$, we show (\cf~\Cref{theorem:main}) that there exists a temperature control~$\eta$ supported in $\omegaup$ such that the corresponding solution~$(\xvec{u}, \theta)$ to \eqref{equation:BoussinesqVelocity} approaches the target at time $t = T$ with respect to a Sobolev norm~$\|\cdot\|$; that is,
	\begin{equation}\label{equation:introtargetcondition}
		\|\xvec{u}(\cdot, T) - \xvec{u}_T\| + \|\theta(\cdot, T) - \theta_T\|  < \varepsilon.
	\end{equation}
	
	\subsection{Highly degenerate controllability problems}\label{subsection:deg} Our goal described above will be achieved as by-product of studying a more degenerate situation. Hereto, let us consider the Boussinesq problem with an additional control acting parallel to gravitation on the velocity:
	\begin{equation}\label{equation:BoussinesqVelocity_vc}
		\begin{gathered}
			\partial_t \xvec{u} - \nu \Delta \xvec{u} + \left(\xvec{u} \cdot \xnab\right) \xvec{u} + \xnab p = (\theta + \mathbb{I}_{\omegaup}\overline{\eta}) \xvec{e}_{\operatorname{grav}} + \xsym{\Phi}_{\operatorname{ext}}, \! \quad \xdiv{\xvec{u}} = 0, \! \quad \xvec{u}(\cdot, 0) = \xvec{u}_0, \\ \partial_t \theta - \tau \Delta \theta + (\xvec{u}\cdot\xnab)\theta = \mathbb{I}_{\omegaup} \eta + \psi_{\operatorname{ext}}, \quad \theta(\cdot, 0) = \theta_0.
		\end{gathered}
	\end{equation}
	We will prove the approximate controllability of \eqref{equation:BoussinesqVelocity_vc} by means of highly degenerate controls~$\eta$ and~$\overline{\eta}$. Hereto, we are going to describe a small number of universally fixed profiles (actuators) $\zeta_1, \dots, \zeta_{6} \in \xLtwo((0,1); \xCinfty(\mathbb{T}^2; \mathbb{R}))$ that vanish away from $\omegaup$. Then, we will seek~$\eta$ and~$\overline{\eta}$ of the form
	\begin{equation}\label{equation:degeneratecontrol}
		\eta(\xvec{x}, t) = \sum_{l = 1}^{6} \gamma_l(t) \zeta_l(\xvec{x}, \gamma(t)), \quad \overline{\eta}(x_1, x_2, t) = \overline{\gamma}(t) \zeta_1(x_2),
	\end{equation}
	where the unknown parameters $\gamma, \overline{\gamma}, \gamma_1, \dots, \gamma_{6} \in \xLtwo((0,T); \mathbb{R})$ resemble the actual controls, and where we use the notation $\zeta_l(x,\gamma(t)) = \zeta_l(\gamma(t))(x)$. We call controls that have the form \eqref{equation:degeneratecontrol} \enquote{finitely decomposable}.
	It is hereby important that the profiles $\zeta_1, \dots, \zeta_{6}$ in \eqref{equation:degeneratecontrol} do not depend on the initial and target states, the viscosity, the thermal diffusivity, the external forces, or the approximation error. Thus, solving this degenerate controllability problem means showing the existence of parameters~$\overline{\gamma}, \gamma_1, \dots, \gamma_{6}$, and~$\gamma$ in the representation \eqref{equation:degeneratecontrol} such that the solution to \eqref{equation:BoussinesqVelocity_vc} satisfies~\eqref{equation:introtargetcondition}. In fact,~$\overline{\gamma}$ will be smooth and satisfy $\operatorname{supp}(\overline{\gamma}) \subset (0, T)$, and it will be possible to take~$\zeta_1$ and~$\zeta_2$ as a smooth single variable functions with $\zeta_2 = \zeta_1'$.  
	
	Moreover, as detailed in \Cref{subsection:conclusion}, the transformation $\theta(x_1,x_2,t) \mapsto \theta(x_1,x_2,t) + \overline{\gamma}(t) \zeta_1(x_2)$ will allow us to
	interchange the controls in \eqref{equation:degeneratecontrol} with
	\begin{equation}\label{equation:degeneratecontrol2}
		\begin{aligned}
				\eta(x_1, x_2, t) & =  \left(\overline{\gamma}'(t) + \gamma_1(t)\right) \zeta_1(x_2)  + \left(\overline{\gamma}(t)\xvec{u}(x_1,x_2, t) \cdot \xsym{e}_{\operatorname{grav}}  + \gamma_2(t)\right) \zeta_2(x_2) \\ 
				& \quad - \tau \overline{\gamma}(t) \zeta_1''(x_2) + \sum_{l = 3}^{6} \gamma_l(t) \zeta_l(x_1, x_2, \gamma(t)),\\
			\overline{\eta}(x_1, x_2, t) & = 0.
		\end{aligned}
	\end{equation}
	Hence, global approximate controllability of \eqref{equation:BoussinesqVelocity_vc} will also be achieved by using merely a finitely decomposable temperature control plus an explicit feedback term.
	More specifically, in \eqref{equation:degeneratecontrol2}, the map~$\xvec{u} \mapsto \overline{\gamma} \zeta_2 \xvec{u} \cdot \xsym{e}_{\operatorname{grav}} $ can be viewed as an explicitly given linear feedback law for the coefficient of the actuator~$\zeta_2$. The main controllability problem described in \Cref{subsection:maincontrollabilityproblem} can be solved by finding a control~$\eta$ of the form stated in \eqref{equation:degeneratecontrol2} for the Boussinesq system~\eqref{equation:BoussinesqVelocity_vc}. 
	
	\begin{rmrk}\label{remark:sidenote}
		The global approximate controllability of \eqref{equation:BoussinesqVelocity} can also be achieved by using merely a finitely decomposable temperature control of the form
		\begin{equation}\label{equation:altnonl}
			\eta(\xvec{x},t) = \overline{\gamma} + \mathbb{I}_{\omegaup}(\gamma_1(t) \zeta_1(\xvec{x}, \gamma(t)) + \dots + \gamma_{6}(t) \zeta_{6}(\xvec{x}, \gamma(t))).
		\end{equation}
		While the first actuator in this representation is not physically localized, this still extends the existing literature on degenerate controls for incompressible fluids. More details are given in \Cref{subsection:conclusion}.
	\end{rmrk}

	Another interest of our work is to relax a topological constraint previously imposed on the control region in~\cite{NersesyanRissel2024}, where the two-dimensional Navier--Stokes system with physically localized finitely decomposable controls is considered. There, in order to act on the velocity-average, the control zone is required to contain two cuts rendering the torus simply-connected (\eg, the union of two strips with linearly independent direction vectors).~In the present article, by exploiting the buoyant force coupling the velocity and temperature in the momentum equation, we are able to take~$\omegaup$ merely as a horizontal strip. This observation indicates that heat effects in the mathematical model might improve certain controllability properties.

	\subsection{Notations} 
	Given any integer $m \geq 0$, several basic $\xLtwo$-based Sobolev spaces of zero average functions and divergence-free vector fields are denoted by
	\begin{equation*}
		\begin{gathered}
			\xH_{\operatorname{avg}} \coloneq \left\{ f \in \xLtwo(\mathbb{T}^2;\mathbb{R}) \, \left| \right. \, \int_{\mathbb{T}^2} \! f(\xvec{x}) \, \xdx{\xvec{x}} = 0 \right\}, \,\, \mathbf{H}_{\operatorname{div}} \coloneq \left\{ \xvec{f} \in \xLtwo(\mathbb{T}^2;\mathbb{R}^2) \, \left| \right. \, \xdiv{\xvec{f}} = 0 \right\},\\
			\bVn{m} \coloneq \xHn{m}(\mathbb{T}^2;\mathbb{R}^2) \cap \mathbf{H}_{\operatorname{div}}, \quad 	\bHn{m} \coloneq \bVn{m} \cap \xH_{\operatorname{avg}}^2, \quad \xHn{m} \coloneq \xHn{m}(\mathbb{T}^2;\mathbb{R}) \cap \xH_{\operatorname{avg}},
		\end{gathered}
	\end{equation*}
	endowed with the norms
	\begin{equation*}
		\|\cdot\|_m \coloneq \sqrt{\sum_{|\xsym{\alpha}| \leq m} \| \partial^{\xsym{\alpha}} \cdot\|_0^2}, \quad \|\cdot\|_0 \coloneq \mbox{ either } \|\cdot\|_{\xLtwo(\mathbb{T}^2;\mathbb{R})} \mbox{ or } \|\cdot\|_{\xLtwo(\mathbb{T}^2;\mathbb{R}^2)},
	\end{equation*}
	and the Lebesgue measure on the torus is assumed normalized: $\smallint_{\mathbb{T}^2} \, \xdx{\xvec{x}} = 1$.
	The symbol~$\mathscr{O}$ refers to the Landau-big-O notation. If not specified otherwise, constants of the form $C > 0$ are generic and may vary during estimates. If $\xvec{x} \in \mathbb{R}^2$ or $\xvec{x} \in \mathbb{T}^2$, we denote $\xvec{x} = (x_1,x_2)$. For a function $J \ni t \mapsto X$, where $X$ is a space of functions of $x$, we occasionally use the notation $f(x,t) = f(t)(x)$.

	\subsection{Main results}\label{subsection:mainresult}

	Throughout, we will take the actuators from a universally fixed collection of so-called transported Fourier modes $\zeta_1, \dots, \zeta_{6} \in \xLtwo((0,1); \xCinfty(\mathbb{T}^2; \mathbb{R}))$ as the building blocks for the controls described in \eqref{equation:degeneratecontrol} and \eqref{equation:degeneratecontrol2}. They will be determined during the proof of \Cref{theorem:localized} in a way only depending on~$\omegaup$ (\Cref{Figure:choices}), and their name is due to the involved composition of usual Fourier modes with certain flow maps. More specifically,~(\cf~\cref{equation:ctrlstrcaux,equation:T})
	\[
		\{\zeta_1, \dots, \zeta_{6}\} = \left\{ \chi, \chi' \right\} \cup \left\{(\xvec{x}, t) \mapsto \chi(\xvec{x}) \widetilde{\zeta}(\xsym{\mathcal{U}}(\xsym{\mathcal{Y}}(\xvec{x}, t, 1),1,\sigma(t))) \, \left| \right. \, \widetilde{\zeta} \in \mathscr{M} \right\},
	\]
	where
	\begin{itemize}
		\item $\mathscr{M} = \left\{\xvec{x} \mapsto \sin(x_1), \, \xvec{x} \mapsto \cos(x_1), \, \xvec{x} \mapsto\sin(x_2), \, \xvec{x} \mapsto\cos(x_2) \right\}$,
		\item $\xsym{\mathcal{Y}}$ and $\xsym{\mathcal{U}}$ are the flow maps of incompressible vector fields introduced in Sections~\Rref{subsection:convection}~and~\Rref{subsection:generatingvectorfield},
		\item $\sigma\colon [0,1] \longrightarrow [0,1]$ is the function defined in \eqref{equation:eqfsig},
		\item $\xvec{x} \mapsto\chi(x_2)$ is a smooth cutoff supported in $\omegaup$ introduced in \Cref{section:preliminaries}.
	\end{itemize}

	\begin{figure}[ht!]\centering\resizebox{0.98\textwidth}{!}{\centering
		\begin{tikzpicture}[node distance=1.2cm]
			\node (s) [blue0] {Fix $\omegaup \subset \mathbb{T}^2$};
			
			\node (c) [blue0, right = 0.6cm of s] {Construct $\zeta_1, \dots, \zeta_{6}$};
			
			\node (p) [blue, right = 0.6cm of c] {Select $\nu > 0$, $\tau > 0$, $T > 0$, $\varepsilon > 0$, $\xsym{\Phi}_{\operatorname{ext}}$, $\psi_{\operatorname{ext}}$, states $(\xvec{u}_0, \theta_0)$, $(\xvec{u}_T, \theta_T)$};
			
			\node (r) [blue, right = 0.6cm of p] {Obtain $\gamma, \overline{\gamma}, \gamma_1,\dots,\gamma_{6}$.};

			\draw [arrow,line width=0.6mm] (s) -- (c);
			\draw [arrow,line width=0.6mm] (c) -- (p);
			\draw [arrow,line width=0.6mm] (p) -- (r);
			
		\end{tikzpicture}}
		\caption{The order in which the building blocks for our control force are chosen.}
		\label{Figure:choices}
	\end{figure}

	The global approximate controllability of~\eqref{equation:BoussinesqVelocity} by means of a physically localized temperature control is now stated as follows; its proof is concluded in \Cref{subsection:conclusion}.
	
	\begin{thrm}\label{theorem:main}
		Let the integer $r\geq 0$, viscosity $\nu > 0$, diffusivity $\tau > 0$, control time $T~>~0$, initial and target states $(\xvec{u}_0, \xvec{u}_T) \in \bHn{{r}}\times\bHn{{r}}$, $(\theta_0, \theta_T) \in \xHn{{r}}\times\xHn{{r}}$, external forces $(\xsym{\Phi}_{\operatorname{ext}}, \psi_{\operatorname{ext}}) \in \xLtwo((0, T); \bHn{{\max\{r,2\}}} \times \xHn{{\max\{r,2\}}})$, and the approximation error $\varepsilon > 0$ be fixed. There exists a control $\eta \in \xCinfty(\mathbb{T}^2\times[0,T];\mathbb{R})$ with $\operatorname{supp}(\eta) \subset \omegaup\times[0,T]$ such that the unique solution
		\begin{gather*}
			\xvec{u} \in \xCzero([0,T];\bVn{{r}})\cap\xLtwo((0,T);\bVn{{r+1}}),\\
			\theta \in \xCzero([0,T];\xHn{{r}}(\mathbb{T}^2; \mathbb{R}))\cap\xLtwo((0,T);\xHn{{r+1}}(\mathbb{T}^2; \mathbb{R}))
		\end{gather*}
		to the Boussinesq problem \eqref{equation:BoussinesqVelocity} satisfies
		\begin{equation}\label{equation:controllabilityproperty}
				\|\xvec{u}(\cdot, T) - \xvec{u}_T\|_{r} + \|\theta(\cdot, T) - \theta_T\|_{r}  < \varepsilon.
		\end{equation}
	\end{thrm}
	
	The proof of \Cref{theorem:main} is a consequence of the following result on the global approximate controllability of the Boussinesq system via physically localized finitely decomposable controls; the proofs of both theorems are completed in \Cref{subsection:conclusion}.
	\begin{thrm}\label{theorem:secondmain}
		Under the same assumptions as in \Cref{theorem:main}, there exist control parameters~$\gamma, \overline{\gamma}, \gamma_1, \dots, \gamma_{6} \in \xLtwo((0, T); \mathbb{R})$ such that the unique solution
		\begin{gather*}
			\xvec{u} \in \xCzero([0,T];\bVn{{r}})\cap\xLtwo((0,T);\bVn{{r+1}}),\\
			\theta \in \xCzero([0,T];\xHn{{r}}(\mathbb{T}^2; \mathbb{R}))\cap\xLtwo((0,T);\xHn{{r+1}}(\mathbb{T}^2; \mathbb{R}))
		\end{gather*}
		to the problem \eqref{equation:BoussinesqVelocity_vc} with $\eta$ and $\overline{\eta}$ of the form \eqref{equation:degeneratecontrol} (or \eqref{equation:degeneratecontrol2}) satisfies~\eqref{equation:controllabilityproperty}.
	\end{thrm}

	\subsection{Related literature and outline}\label{subsection:literaturecontr}
	
	The recent decades have seen various studies concerned with controllability properties of fluids exhibiting Boussinesq heat effects.~A natural question -- which remains in most respects widely open -- is whether these systems can be steered to a desired state by merely applying external cooling or heating in a possibly small subset of the domain. Let us subdivide previous research in that or related directions into three categories.

	1) When the controls enter all the evolution equations (or boundary conditions) for the velocity and the temperature, there exists a rich body of literature. For instance, several authors have invoked linearization techniques and then studied the controllability of linear problems; \eg, see \cite{FursikovImanuvilov1998, Guerrero2006} and the references therein, where duality arguments,~Carleman estimates, and local inversion theorems play crucial roles. For the local exact controllability of the Navier--Stokes system, we refer to \cite{Fernandez-CaraGuerreroImanuvilovPuel2004} and the references therein. However, due to the nonlinear effects, the aforementioned results require initial and target profiles that are sufficiently close in the state space. As a way to remove such smallness constraints, hence to achieve global controllability properties, it was shown that Coron's return method (\cf~\cite{Coron2007,CoronFursikov1996,Coron96}) can be applied: \eg, in \cite{Chaves-SilvaEtal2023,FernandezCaraSantosSouza2016,FursikovImanuvilov1999} for both viscous and inviscid Boussinesq systems. In this context, one should also name a famous open problem posed by J.-L. Lions on the global approximate controllability of the Navier--Stokes equations, in bounded $2$D and $3$D domains, with the no-slip boundary condition (\cf~\cite{LionsJL1991,CoronMarbachSueurZhang2019}). A version of this problem has been resolved in \cite{CoronMarbachSueur2020} for the case of Navier slip-with-friction boundary conditions. Concerning the related subject of stabilization, we point out that small-time local stabilization of the $2$D Navier--Stokes system has been shown in \cite{Shengquan2023}, and small-time global stabilization has been achieved in \cite{CoronShengquan2021} for the viscous Burgers equation with three scalar controls.
	
	2) When the controls only act in few components of the considered system, less is known regarding its controllability. For the three-dimensional Navier--Stokes equations with the no-slip boundary condition, the work~\cite{CoronLissy2014} demonstrates the local exact null controllability with controls vanishing in two components; see also \cite{CoronGuerrero2009} for a prior result with one vanishing component. For the Boussinesq system in~$N$ dimensions~($N \geq 2$) in domains with boundaries, the local exact controllability to certain trajectories has been shown by using controls that act only in~$N-1$ directions, \eg, in \cite{Fernandez-CaraGuerreroImanuvilovPuel2006,Carreno2012,Montoya2020}.  In general, exact controllability with degenerate controls cannot be expected; \eg, see \cite{Shirikyan2008} for a negative result on the exact controllability of the incompressible Euler equations with finite-dimensional controls.
	
	3) Another important class of controls are those resembling finite combinations of fixed actuators. For the Navier--Stokes equations, finite-dimensional controls can be constructed via the Agrachev-Sarychev method \cite{AgrachevSarychev2006} and its refinements, for instance, as provided in \cite{Shirikyan2007,Nersesyan2021}. However, in these references, the controls are not physically localized (they act everywhere in the torus); whether their localization in space is possible constitutes an open problem due to Agrachev \cite{Agrachev2014}. To achieve also physical localization, we replaced in~\cite{NersesyanRissel2024} the notion of finite-dimensional controls by that of finitely decomposable ones, where some of the universally fixed actuators depend on time. Here, as a byproduct, we extend these results to the $2$D Boussinesq system. By exploiting the temperature coupling, this leads now to additional improvements such as reduced constraints on the control region (in \cite{NersesyanRissel2024}, the controls are not supported in a horizontal strip).
	
	This article contributes a first global (large data) controllability result for the Boussinesq system via physically localized controls acting only in the temperature. Even more, the controls can be chosen finitely decomposable (of the type \eqref{equation:degeneratecontrol}) if one admits a one-dimensional control in the second component of the velocity problem. But, up to an explicit feedback supported in $\omegaup$, we can also steer the system merely with a physically localized and finitely decomposable temperature control (\cf~\eqref{equation:degeneratecontrol2}). Moreover, we allow prescribed external forces in the right-hand sides of the Boussinesq problem. The main ingredients of our approach are as follows.
	
	a) A multi-stage scaling procedure that combines two mechanisms: i) controlling the fluid's vorticity, but ignoring the temperature;~ii) steering the temperature without influencing the vorticity much. To this end, we develop ideas from \cite{NersesyanRissel2024,BoulvardGaoNersesyan2023}, and also involve a version of Coron's return method and hydrodynamic scaling from \cite{Coron96}. 

	b) The physical localization of $\zeta_1, \dots, \zeta_{6}$ is achieved via careful rearrangements of integrals that represent solutions to transport equations; in that way, we further develop several of our ideas from \cite{NersesyanRissel2024}.

	\paragraph{Outline of the paper.} 
	As described in \Cref{section:proofmainresult}, the proofs of Theorems~\Rref{theorem:main} and~\Rref{theorem:secondmain} are reduced to controlling the vorticity, the temperature, and the average velocity.~These sub-goals are achieved by means of the following main steps (\cf~\Cref{Figure:at}). {\it A)} The temperature can be controlled without significantly changing the vorticity; see \Cref{theorem:main_smalltimelargecontrol}. Hereto, preliminary constructions are presented in \Cref{section:preliminaries}, finitely decomposable controls (possibly supported everywhere) are obtained in \Cref{section:degencontrols} for linear transport equations, a localization procedure is carried out in~\Cref{section:localizing}, and a hydrodynamic scaling is discussed in \Cref{subsection:wellposedness}. 
	{\it B)} As detailed in \Cref{theorem:propvortcontrl}, the vorticity can be controlled through well-prepared initial conditions. 
	{\it C)} The previous arguments are put together in the proof of \Cref{theorem:control_vorttemp}, which concerns a vorticity-temperature formulation for the Boussinesq system. Finally, the main results are concluded in \Cref{subsection:conclusion}.
	
	\begin{figure}[ht!]\centering\resizebox{0.75\textwidth}{!}{\centering
			\begin{tikzpicture}[node distance=0.5cm]
				\node (A) [B00] {\Large\Cref{theorem:control_vorttemp}\\ {\it (Controlling vorticity and temperature in arbitrary time)} \\ \vspace{2pt} Proof: \Cref{subsection:vortcontrlglb}};
				\node (B) [B000, above = 0.5cm of A] {\Large Theorems~\Rref{theorem:main} and~\Rref{theorem:secondmain} \\ {\it (Main results)} \\ \vspace{2pt} Proofs: \Cref{subsection:conclusion}};

				\node (AA) [below = 1.3cm of A] {};
				
				\node (C) [B0, right = 0.1cm of AA] {\Large\Cref{theorem:propvortcontrl} \\ {\it (Large initial data control, short time)} \\ \vspace{2pt} Proof: \Cref{subsection:vortcontrlglb}};
				\node (D) [B0, left = 0.1cm of AA] {\Large\Cref{theorem:main_smalltimelargecontrol} \\ {\it (Large additive control, short time)} \\ \vspace{2pt}Proof: \Cref{subsection:wellposedness}};
				
				\node (E) [B11, below = 1.3cm of AA] {\Large\Cref{theorem:localized} \\ {\it (Localized control for linear problem)} \\ \vspace{2pt} Proof: \Cref{section:localizing}};

				\draw [arrow,line width=0.5mm] (A) -- (B);
				\draw [arrow,line width=0.5mm] (D) -- (A);
				\draw [arrow,line width=0.5mm] (C) -- (A);
				\draw [arrow,line width=0.5mm] (E) -- (D);
		\end{tikzpicture}}
		\caption{Subdivision of the proofs for Theorems~\Rref{theorem:main} and~\Rref{theorem:secondmain}.}
		\label{Figure:at}
	\end{figure}

	\section{Controls for linear transport problems}
	This section concerns linear transport problems driven by finitely decomposable controls. After several preliminary constructions, we will first obtain in \Cref{section:degencontrols} finitely decomposable controls whose spatial support can be contained in any subset of $\mathbb{T}^2$. Subsequently, in \Cref{section:localizing}, we will construct finitely decomposable controls that are physically supported in~$\omegaup$.
	
	\subsection{Preliminary constructions}\label{section:preliminaries}
	This section aims to collect several definitions and constructions that are required subsequently.
	
	\subsubsection{Partition of the torus}\label{subsection:partition}
	Select $0 < H_1 < H_2 < 2\pi$ in a way that $\mathbb{T} \times [H_1, H_2] \subset \omegaup$.
	Then, a number~$K \in \mathbb{N}$ is chosen such that
	\begin{equation*}
		l_K \coloneq \frac{8\pi}{3K} <  \frac{H_2 - H_1}{3}.
	\end{equation*}
	As illustrated in \Cref{Figure:Covering}, the torus $\mathbb{T}^2$ may thus be covered by a family of overlapping strips $(\mathcal{O}_i)_{i\in\{1,\dots,K\}}$ having fixed overlap length $l_K /4$ and being translated copies of the reference strip
	\[
		\mathcal{O} \coloneq \mathbb{T} \times (H_1+l_K, H_1+2l_K) \subset \omegaup.
	\]
	For definiteness, let us take
	\[
		\mathcal{O}_i \coloneq \mathbb{T} \times \left(\frac{3(i-1)l_{K}}{4}, \frac{3(i-1)l_{K}}{4} + l_{K}
	\right), \quad i \in \{1,\dots,K\}.
	\]
	On this basis, a reference cutoff function $\chi \in \xCinfty(\mathbb{T}^2;[0,1])$ with~$\operatorname{supp}(\chi) \subset \mathcal{O}$ is specified via
	\begin{equation}\label{equation:Definition_chi}
		\chi(\xvec{x}) = \chi(x_2)  \coloneq \mu(x_2-H_1-l_K), \quad \xvec{x}\in\mathbb{T}^2,
	\end{equation}
	where $\mu \in \xCinfty(\mathbb{T};[0,1])$ satisfies
	\begin{equation}\label{equation:partitionofunity}
		\begin{gathered}
			\operatorname{supp}(\mu) \subset (0, l_K), \quad \forall x \in (0,l_K/4)\colon  \mu(x) + \mu(x + 3l_K/4) = 1,\\
			\mu(s) = 1 \iff s \in \left[l_K/4, 3l_K/4\right].
		\end{gathered}
	\end{equation}
	
	\begin{figure}[ht!]
		\centering
		\resizebox{0.55\textwidth}{!}{
			\begin{tikzpicture}
				\clip(-0.8,-0.5) rectangle (5.85,5.2);

				\draw[line width=0mm, color=white, fill=FireBrick!10] plot[smooth cycle] (0,0.9) rectangle (5,1.9);

				\draw[line width=0.1mm, color=Black, opacity=0, postaction={pattern=dots,opacity=0.8}] plot[smooth cycle] (0,0) rectangle (5,0.25);
				\draw[line width=0.1mm, color=Black, opacity=0, postaction={pattern=dots,opacity=0.8}] plot[smooth cycle] (0,4.5) rectangle (5,4.75);
				
				\draw[line width=0.1mm, color=Black] plot[smooth cycle] (0,4.5) rectangle (5,4.5);
				
				\draw[line width=0.1mm, color=Black] plot[smooth cycle] (0,0) -- (5,0);
				\draw[line width=0.1mm, dashed, color=Black] plot[smooth cycle] (0,0.25) -- (5,0.25);
				\draw[line width=0.1mm, dotted, color=Black] plot[smooth cycle] (0,0.75) -- (5,0.75);
				\draw[line width=0.1mm, color=Black] plot[smooth cycle] (0,1) -- (5,1);
				\draw[line width=0.1mm, dotted, color=Black] plot[smooth cycle] (0,1.75) -- (5,1.75);
				\draw[line width=0.1mm, dashed, color=Black] plot[smooth cycle] (0,1.5) -- (5,1.5);
				\draw[line width=0.1mm, dashed, color=Black] plot[smooth cycle] (0,2.5) -- (5,2.5);
				\draw[line width=0.1mm, color=Black] plot[smooth cycle] (0,2.25) -- (5,2.25);
				\draw[line width=0.1mm, color=Black] plot[smooth cycle] (0,3.25) -- (5,3.25);
				\draw[line width=0.1mm, dotted, color=Black] plot[smooth cycle] (0,3) -- (5,3);
				\draw[line width=0.1mm, dotted, color=Black] plot[smooth cycle] (0,4) -- (5,4);
				\draw[line width=0.1mm, dashed, color=Black] plot[smooth cycle] (0,3.75) -- (5,3.75);
				\draw[line width=0.1mm, dashed, color=Black] plot[smooth cycle] (0,4.75) -- (5,4.75);
				\draw[line width=0.1mm, color=Black] plot[smooth cycle] (0,4.5) -- (5,4.5);

				\draw [decorate,
				decoration = {calligraphic brace}, line width = 0.3mm] (-0.15,0) --  (-0.15,1);
				\coordinate[label=below:\scriptsize$\mathcal{O}_1$] (A) at (-0.4,0.73);

				\draw [decorate,
				decoration = {calligraphic brace}, line width = 0.3mm] (-0.02,0.75) --  (-0.02,1.75);
				\coordinate[label=below:\scriptsize$\mathcal{O}_2$] (A) at (-0.27,1.48);
				
				\draw [decorate,
				decoration = {calligraphic brace}, line width = 0.3mm] (-0.15, 1.5) --  (-0.15,2.5);
				\coordinate[label=below:\scriptsize$\mathcal{O}_3$] (A) at (-0.4,2.23);

				\draw [decorate,
				decoration = {calligraphic brace}, line width = 0.3mm] (5.13, 3.25) --  (5.13, 2.25);
				\coordinate[label=below:\scriptsize$\mathcal{O}_4$] (A) at (5.42, 2.98);

				\draw [decorate,
				decoration = {calligraphic brace}, line width = 0.3mm] (5.02, 4) --  (5.035, 3);
				\coordinate[label=below:\scriptsize$\mathcal{O}_5$] (A) at (5.31, 3.73);

				\draw [decorate,
				decoration = {calligraphic brace}, line width = 0.3mm] (5.13, 4.75) --  (5.13, 3.75);
				\coordinate[label=below:\scriptsize$\mathcal{O}_6$] (A) at (5.42, 4.48);

				\coordinate[label=below:\color{FireBrick}\scriptsize Reference strip $\mathcal{O}$] (B) at (2.39, 1.57);
			\end{tikzpicture}
		}
		\caption{An exemplary open covering of $\mathbb{T}^2$ by $K = 6$ overlapping strips $\mathcal{O}_1, \dots, \mathcal{O}_6$, the boundaries of which are in an alternating way depicted as solid, dashed, and dotted lines. The overlapping region due to vertical periodicity is highlighted by a dotted pattern. The reference strip $\mathcal{O}$ contained inside the control region is displayed as a~(red) filled rectangle.}
		\label{Figure:Covering}
	\end{figure}

	\subsubsection{Convection strategy}\label{subsection:convection}
	For the purpose of physically localizing the support of controls for linear transport problems in the control zone~$\omegaup$, a special spatially constant vector field is now constructed. One key property of this vector field is that its integral curves pass through~$\omegaup$ in a specific way. This profile will also be used as a reference trajectory for Coron's return method~(\cf~\Cref{theorem:SmallTimeConvergenceToLinearized} below); we already introduced similar constructions in~\cite{NersesyanRissel2024}.
	To begin with, the reference time interval~$[0,1]$ is equidistantly partitioned into subintervals  of length~$T^{\Delta} > 0$ by means of
	\begin{equation}\label{equation:timepartition}
		0 < t^0_c < t^1_a < t^1_b < t^1_c < t^2_a < t^2_b < t^2_c < \dots < t^K_a < t^K_b < t^K_c < 1,
	\end{equation}
	where $	t_c^0 = t^i_a - t^{i-1}_c = t^i_c-t^i_b = t^i_b-t^i_a = 1 - t^K_c = T^{\Delta}$
	for all $i \in \{1,\dots,K\}$.
	\begin{thrm}\label{theorem:convection}
		There exists a function $\overline{\xvec{y}} = (0, \overline{y}_2) \in \xCinfty_0((0,1); \{0\}\times\mathbb{R})$ satisfying $\operatorname{supp}(\overline{\xvec{y}}) \subset [t^0_c, t^K_c]$ and the properties
		\begin{gather*}
			\forall \xvec{x} \in \mathbb{T}^2\colon \xsym{\mathcal{Y}}(\xvec{x},0,1) = \xvec{x},\\
			\forall i \in \{1,\dots,K\},\,  \forall t \in [t^i_a, t^i_b] \colon \,\mathcal{O} = \xsym{\mathcal{Y}}(\mathcal{O}_i, 0, t) \coloneq \left\{ \xsym{\mathcal{Y}}(\xvec{x}, 0, t) \, | \, \xvec{x} \in \mathcal{O}_i \right\},
		\end{gather*}
		where $\xsym{\mathcal{Y}}$ denotes the flow of $\overline{\xvec{y}}$ obtained by solving
		\begin{equation}\label{equation:Y}
			\xdrv{\xsym{\mathcal{Y}}}{t}(\xvec{x}, s, t) = \overline{\xvec{y}}(t), \quad \xsym{\mathcal{Y}}(\xvec{x}, s, s) = \xvec{x}.
		\end{equation}
	\end{thrm}
	\begin{proof}
		The argument goes along the lines of \cite[Theorem 3.3]{NersesyanRissel2024}.
		First, a family of functions $(\beta_i)_{i\in\{1,\dots,K\}} \subset \xCinfty_0((0, T^{\Delta});\mathbb{R})$ is chosen such that
		\[
			\mathcal{O} = \mathcal{O}_i + \xvec{e}_{\operatorname{grav}}\int_0^{T^{\Delta}} \beta_i(s) \, \xdx{s} \coloneq \left\{ \xvec{x} + \xvec{e}_{\operatorname{grav}}\int_0^{T^{\Delta}} \beta_i(s) \, \xdx{s} \, \Bigg| \, \xvec{x} \in \mathcal{O}_i \right\}
		\]
		holds for each $i\in\{1,\dots,K\}$. Then, for each $i \in \{1, \dots, K\}$, an auxiliary profile $\xvec{h}_i \in \xCinfty_0((0, 3T^{\Delta});\mathbb{R}^2)$ is defined by means of
		\[
		\xvec{h}_i(t) = \begin{cases}
			\beta_i(t) \xvec{e}_{\operatorname{grav}} & \mbox{ if } t \in [0, T^{\Delta}],\\
			\xsym{0} & \mbox{ if } t \in (T^{\Delta}, 2T^{\Delta}),\\
			-\beta_i(t - 2T^{\Delta}) \xvec{e}_{\operatorname{grav}} & \mbox{ if } t \in [2T^{\Delta}, 3T^{\Delta}].
		\end{cases}
		\]
		Finally, a function $\overline{\xvec{y}}$ with the desired properties is given by
		\[
		\overline{\xvec{y}}(t) \coloneq \begin{cases}
			\xsym{0} & \mbox{ if } t \in [0, t^0_c] \cup [t^K_c, 1],\\
			\xvec{h}_i(t-(3i-2)T^{\Delta}) & \mbox{ if } t \in (t^{i-1}_c, t^i_c) \mbox{ for } i \in \{1,\dots, K\}.
		\end{cases}
		\]
	\end{proof}
	
	\subsubsection{A generating vector field}\label{subsection:generatingvectorfield}
	Let us setup some terminology that will be required later in \Cref{section:degencontrols}.~A key ingredient is the following observability notion, which has been introduced in~\cite{KuksinNersesyanShirikyan2020} for the study of mixing properties of randomly forced PDEs.
	
	\begin{dfntn}\label{definition:observablefamily}
		Given any $T > 0$ and $n \in \mathbb{N}$, a family $(\phi_j)_{j \in \{1,\dots,n\}} \subset \xLtwo((0,T);\mathbb{R})$ is said to be observable if
		\begin{gather*}
			\forall \, \xI \in  \{\mbox{subintervals of } (0,T)\}, \, \forall \, (a_j)_{j \in \{1,\dots,n\}} \subset \xCone(\xI; \mathbb{R}), \, \forall \, b \in \xCzero(\xI;\mathbb{R})\colon \\
			b + \sum_{j=1}^n a_j \phi_j = 0 \, \mbox{ in } \xLtwo(\xI;\mathbb{R}) \quad 	\Longrightarrow \quad \forall t \in \xI\colon b(t) = a_1(t) = \cdots = a_n(t) = 0.
		\end{gather*}
	\end{dfntn}
	
	\begin{rmrk}
		Observable families in the sense of \Cref{definition:observablefamily} can be constructed in an explicit way; \eg, see \cite[Section 3.3]{Nersesyan2021} or \cite{NersesyanRissel2024} for more details. 
	\end{rmrk}
	
	Let~$(\phi_l)_{l \in \{1, \dots, 4\}} \subset \xLtwo((0, 1); \mathbb{R})$ be observable, and take~$\phi \in \xCone([0, 1];\mathbb{R})$ such that~$\phi(t) = 0$ if and only if~$t = 1$.
	Furthermore, define the family
	\begin{equation*}
		(\psi_l)_{l \in \{1, \dots, 4\}} \subset \xWn{{1,2}}((0, 1);\mathbb{R}), \quad \forall l \in \{1,\dots, 4\} \colon \, \psi_{l}(t) \coloneq \phi(t) \int_0^t \phi_{l}(s) \, \xdx{s}.
	\end{equation*}
	Then, a \enquote{generating} divergence-free vector field $\overline{\xvec{u}} \in \xWn{{1,2}}((0, 1);\xCinfty(\mathbb{T}^2;\mathbb{R}^2))$ is given via (\cf~\cite[Section 3.4]{NersesyanRissel2024})
	\begin{equation*}
		\overline{\xvec{u}}(\xvec{x}, t) \coloneq  \begin{bmatrix}
			\psi_1(t) \sin(x_2) + \psi_2(t) \cos(x_2) \\ \psi_3(t) \sin(x_1) + \psi_4(t) \cos(x_1)
		\end{bmatrix},
	\end{equation*}
	and its flow $\xsym{\mathcal{U}}$ solves
	\[
		\xdrv{\xsym{\mathcal{U}}}{t}(\xvec{x}, s, t) = \overline{\xvec{u}}(\xsym{\mathcal{U}}(\xvec{x}, s, t), t), \quad \xsym{\mathcal{U}}(\xvec{x}, s, s) = \xvec{x}.
	\]
	Here, the term~\enquote{generating} expresses that such vector fields are able to induce all desired directions via finite-dimensional controls, as demonstrated in, \eg, \cite{NersesyanRissel2024, Nersesyan2021}.
	
	\subsection{Non-localized degenerate controls}\label{section:degencontrols}
  	Let $\overline{\xvec{y}}$ with flow $\xsym{\mathcal{Y}}$ and $\overline{\xvec{u}}$ with flow $\xsym{\mathcal{U}}$ be as introduced in Sections~\Rref{subsection:convection} and~\Rref{subsection:generatingvectorfield}, respectively. To begin with, we state the following result whose proof is straightforward.
    
    \begin{lmm}\label{lemma:pert}
    	Given $m \in \mathbb{N}$ and $\xvec{b} \in \xCzero([0,1]; \xCn{m}(\mathbb{T}^2; \mathbb{R}^2))$, the linear operator which associates to any prescribed force~$g \in \xLtwo((0,1); \xHn{{m}}(\mathbb{T}^2;\mathbb{R}))$ the unique solution $v \in \xCzero([0,1]; \xHn{{m}}(\mathbb{T}^2;\mathbb{R}))$ to the transport equation
    	\[
    		\partial_t v + (\xvec{b} \cdot \xnab) v = g, \quad v(\cdot, 0) = 0
    	\]
    	maps continuously from $\xLtwo((0,1); \xHn{{m}}(\mathbb{T}^2;\mathbb{R}))$ to $\xCzero([0,1]; \xHn{{m}}(\mathbb{T}^2;\mathbb{R}))$.
    \end{lmm}

    Next, a recent result from \cite{NersesyanRissel2024} is recalled concerning finite-dimensional controls for linear transport equations with a generating drift as defined above in \Cref{subsection:generatingvectorfield}. Hereby, as already anticipated in \Cref{subsection:mainresult}, we denote the four-dimensional function space
    \begin{equation}\label{equation:H0}
    	\mathcal{H}_0 = \operatorname{span}_{\mathbb{R}} \mathscr{M}, \quad \mathscr{M} = \{ \sin(x_1), \,  \cos(x_1), \, \sin(x_2), \, \cos(x_2) \}.
    \end{equation}
    The following theorem can be verified by adopting \cite[Section 2.3]{Nersesyan2021} (written there for $3$D) to the $2$D case; see \cite{NersesyanRissel2024} for specific details. Hereby, if desired, the control can be selected in continuous dependence on $(v_1, \theta_1)$ following a compactness argument as explained in \cite[Proof of Theorem 2.3]{Nersesyan2021} or \cite[Proposition 2.6]{KuksinNersesyanShirikyan2020}.
    \begin{thrm}\label{theorem:ControlLinearUncoupledNonlocalized}
    	Given any $m \in \mathbb{N}$,~$z_{1} \in \xHn{{m}}$, and $\varepsilon > 0$, there exists $g \in \xLtwo((0, 1); \mathcal{H}_0)$ such that the unique solution $z \in \xCzero([0, 1];\xHn{{m}})\cap\xWn{{1,2}}((0, 1);\xHn{{m-1}})$ to the linear transport problem
    	\begin{equation}\label{equation:fce}
    		\partial_t z + (\overline{\xvec{u}} \cdot \xnab) z = g,
    		\quad z(\cdot, 0) = 0
    	\end{equation}
    	satisfies
    	\begin{equation*}
    		\|z(\cdot, 1) - z_{1}\|_{m}  < \varepsilon.
    	\end{equation*}
    	In addition, given a bounded subset $\xB \subset \xHn{{m}}$ and any $\varepsilon > 0$, there exists a continuous linear operator~$\mathcal{C}_{\varepsilon}$ which assigns to each $z_1 \in \xB$ a control $g \in \xLtwo((0,T);\mathcal{H}_0)$ such that the corresponding solution $z$ to \eqref{equation:fce} satisfies~$\|z(\cdot, 1)-z_1\|_{m-1} < \varepsilon \|z_1\|_m$.
    \end{thrm}

	The basic intuition for the above theorem is to utilize the convection term $(\overline{\xvec{u}}\cdot\nabla)z$ and the observable structure of $\overline{\xvec{u}}$ in order to create higher frequencies from the low frequency Fourier modes that are provided by the control $g$.
    
    Now, we consider a linearized Boussinesq system driven by degenerate controls that are potentially supported everywhere in $\mathbb{T}^2$. To prepare the localization procedure carried out later in \Cref{section:localizing}, convection will now be realized along $\overline{\xvec{y}}$, and a finite family of transported Fourier modes $\{\widetilde{\zeta}_1, \dots, \widetilde{\zeta}_{4}\}\subset \xLtwo((0,1); \xCinfty(\mathbb{T}^2; \mathbb{R}))$ will be involved instead of $\mathcal{H}_0$; namely, we enumerate
    \begin{equation}\label{equation:T}
    	 \{\widetilde{\zeta}_1, \dots, \widetilde{\zeta}_{4}\} = \Big\{(\xvec{x}, t) \mapsto \widehat{\zeta}(\xsym{\mathcal{U}}(\xsym{\mathcal{Y}}(\xvec{x}, t, 1),1,t)) \, \Big| \, \widehat{\zeta} \in \mathscr{M} \Big\}.
    \end{equation}
    The definition in \eqref{equation:T} is motivated by the proof of the following theorem.

	\begin{thrm}\label{theorem:ControlLinearNonlocalized}
		For any $m \in \mathbb{N}$, $\theta_1 \in \xHn{{m}}(\mathbb{T}^2;\mathbb{R})$, and $\varepsilon > 0$, there exist control parameters $\widetilde{\alpha}_1, \dots, \widetilde{\alpha}_{4} \in \xLtwo((0,1); \mathbb{R})$
		such that the unique solution
        \begin{gather*}
            \theta \in \xCzero([0,1];\xHn{{m}}(\mathbb{T}^2;\mathbb{R}))\cap\xWn{{1,2}}((0,1);\xHn{{m-1}}(\mathbb{T}^2;\mathbb{R}))
        \end{gather*}
		to the linear problem 
		\begin{equation}\label{equation:BoussinesqLinearizedNonlocalized}
			\begin{gathered}
				\partial_t \theta + (\overline{\xvec{y}} \cdot \xnab) \theta = g \coloneq \sum_{l=1}^{4} \widetilde{\alpha}_l \widetilde{\zeta}_l, \quad \theta(\cdot, 0) = 0
			\end{gathered}
		\end{equation}
		satisfies 
		\begin{equation}\label{equation:BoussinesqNonlocalizedCond}
			\|\theta(\cdot,1)-\theta_1\|_{m} < \varepsilon
		\end{equation}
		and the control's space-time average vanishes:
		\begin{equation}\label{equation:gtimespaceaverage}
			\int_0^1 \int_{\mathbb{T}^2} g(\xvec{x}, s) \, \xdx{\xvec{x}} \xdx{s} = 0.
		\end{equation}
		Moreover, given a bounded subset $\xB \subset \xHn{{m}}(\mathbb{T}^2;\mathbb{R})$, there exists a continuous linear operator assigning to each $\theta_1 \in \xB$ a choice of $\widetilde{\alpha}_1, \dots, \widetilde{\alpha}_{4} \in \xLtwo((0,1); \mathbb{R})$ such that the solution $\theta$ to \eqref{equation:BoussinesqLinearizedNonlocalized} satisfies $\|\theta(\cdot,1)-\theta_1\|_{m-1} < \varepsilon$.
	\end{thrm}
	
	\begin{proof}
		By resorting to \Cref{theorem:ControlLinearUncoupledNonlocalized}, we select a control~$\overline{g} \in \xLtwo((0,1); \mathcal{H}_0)$ which resolves the controllability problem
		\begin{equation*}
			\partial_t \overline{\theta} + (\overline{\xvec{u}} \cdot \xnab) \overline{\theta} = \overline{g},
			\quad \overline{\theta}(\cdot, 0) = 0, \quad \|\overline{\theta}(\cdot,1) - \theta_1\|_{m}  < \varepsilon.
		\end{equation*}
		In particular, by \Cref{theorem:ControlLinearUncoupledNonlocalized} and the definition of $\mathcal{H}_0$ in \eqref{equation:H0}, there exist $\widetilde{\alpha}_1, \dots, \widetilde{\alpha}_{4} \in \xLtwo((0,1); \mathbb{R})$ such that
		\[
			\overline{g}(\xvec{x},t) = \widetilde{\alpha}_1(t) \sin(x_1) + \widetilde{\alpha}_2(t) \cos(x_1) + \widetilde{\alpha}_3(t) \sin(x_2) + \widetilde{\alpha}_4(t) \cos(x_2).
		\]
		If $\xB \subset \xHn{{m}}(\mathbb{T}^2;\mathbb{R})$ is bounded with $\theta_1 \in \xB$, we can assume that $\overline{g}$ is the image of~$\theta_1$ under a bounded linear operator $\mathcal{C}_{\varepsilon}$, as provided by \Cref{theorem:ControlLinearUncoupledNonlocalized}.
		Since~$\overline{\xvec{y}}$ is given by~\Cref{theorem:convection}, and the associated flow~$\xsym{\mathcal{Y}}$ is, like~$\xsym{\mathcal{U}}$, volume preserving, the functions
		\begin{equation*}
			\theta(\xvec{x}, t) \coloneq \int_0^t g(\xsym{\mathcal{Y}}(\xvec{x}, t, s), s) \, \xdx{s}, \quad g(\xvec{x}, t) \coloneq \overline{g}(\xsym{\mathcal{U}}(\xsym{\mathcal{Y}}(\xvec{x}, t, 1),1,t), t)
		\end{equation*}
		satisfy
		\begin{equation*}
			\partial_t \theta + (\overline{\xvec{y}} \cdot \xnab) \theta = g, 
			\quad  \theta(\cdot, 0) = 0	, \quad  \theta(\cdot, 1) = \overline{\theta}(\cdot, 1),
		\end{equation*}
		and it holds
		\[
			 \int_{\mathbb{T}^2} (g - \overline{g})(\xvec{x}, s) \, \xdx{\xvec{x}} = 0
		\]
		for almost all~$s \in [0,1]$. 
	\end{proof}

    \subsection{Localized degenerate controls}\label{section:localizing}
    The controls obtained via \Cref{theorem:ControlLinearNonlocalized} are now transformed into controls that are physically localized in~$\omegaup$. These new controls will be given in terms of $6$ fixed profiles that are independent of all data -- except the control region~$\omegaup$ -- imposed in \Cref{theorem:main}. 
    Hereto, let us recall that $\overline{\xvec{y}} = (0, \overline{y}_2)$ from \Cref{theorem:convection} only depends on time and is compactly supported in $(0,1)$.
    
     \begin{thrm}\label{theorem:localized}
    	There exist profiles $\zeta_1, \dots, \zeta_{6} \in \xLtwo((0,1); \xCinfty(\mathbb{T}^2; \mathbb{R}))$ that depend only on the control region $\omegaup$, and for which the following statement holds. Given any $m \in \mathbb{N}$, $\theta_1 \in \xHn{{m}}$, and $\varepsilon > 0$,
    	there are parameters $\alpha_{1}, \dots, \alpha_{6} \in \xLtwo((0,1); \mathbb{R})$ such that the unique solution
    	\begin{gather*}
    		\Theta \in \xCzero([0,1];\xHn{{m}})\cap\xWn{{1,2}}((0,1);\xHn{{m-1}})
    	\end{gather*}
    	to the linear problem
    	\begin{equation}\label{equation:BoussinesqLinearizedlocalized}
    		\partial_t \Theta + (\overline{\xvec{y}} \cdot \xnab) \Theta = \mathbb{I}_{\omegaup} \eta,
    		\quad \Theta(\cdot, 0) = 0,
    	\end{equation}
    	satisfies
    	\begin{equation}\label{equation:LinApproxEst}
    		\|\Theta(\cdot,1)-\theta_1\|_{m} < \varepsilon,
    	\end{equation}
	    where
	    \begin{equation}\label{equation:avrg}
	    	\eta\coloneq \sum_{l=1}^{6} \alpha_{l} \zeta_l \in \xLtwo((0,1); \xCinfty(\mathbb{T}^2; \mathbb{R})), \quad \int_{\mathbb{T}^2} \eta(\xvec{x}, \cdot) \, \xdx{\xvec{x}} = 0 \, \mbox{ a.e.}.
	    \end{equation}
    	In addition, given a bounded subset $\xB \subset \xHn{{m}}$ and $\varepsilon > 0$, there exists a continuous linear operator $\mathcal{C}_{\varepsilon}\colon \xHn{{m}} \longrightarrow \xLtwo((0,1); \mathbb{R})^6$ assigning to each $\theta_1 \in \xB$ a choice of parameters $\alpha_{1}, \dots, \alpha_{6}$ such that the solution $\Theta$ to \eqref{equation:fce} satisfies~$\|\Theta(\cdot,1)-\theta_1\|_{m-1} < \varepsilon \|\theta_1\|_m$.
    \end{thrm}
    
    \begin{proof} 
    Let $\widetilde{\alpha}_{1}, \dots, \widetilde{\alpha}_{4} \in \xLtwo((0,1); \mathbb{R})$, and the corresponding solution $\theta$ to \eqref{equation:BoussinesqLinearizedNonlocalized}, be fixed by applying \Cref{theorem:ControlLinearNonlocalized} with target temperature $\theta_1 \in \xHn{{m}}(\mathbb{T}^2;\mathbb{R})$ such that 
    \[
    	\|\theta(\cdot,1)-\theta_1\|_{m} < \varepsilon.
    \]
    In the case that $\theta_1$ is from $\xB$, we select $\widetilde{\alpha}_{1}, \dots, \widetilde{\alpha}_{4} \in \xLtwo((0,1); \mathbb{R})$ as the image of~$\theta_1$ under a bounded linear operator, while ensuring $\|\theta(\cdot,1)-\theta_1\|_{m-1} < \varepsilon \|\theta_1\|_m$.

    \paragraph{Step 1. Definition of a localized control.}
    Let us recall from \eqref{equation:timepartition} the partition of the reference time interval $(0,1)$ with uniform spacing
    \[	
    	0 < t^0_c < t^1_a < t^1_b < t^1_c < t^2_a < t^2_b < t^2_c < \dots < t^K_a <	t^K_b < t^K_c < 1.
    \]
    The force~$g \coloneq \sum_{l=1}^{4} \widetilde{\alpha}_{l} \widetilde{\zeta}_l$, obtained above via \Cref{theorem:ControlLinearNonlocalized}, is now transformed into a new control~$f$ supported in~$\omegaup$. More specifically, we define
    \begin{equation}\label{equation:eqf}
    	f(\xvec{x}, t) \coloneq \chi(x_2) \sum_{k=1}^K  \frac{1}{t_b^k-t_a^k} \mathbb{I}_{[t_a^k, t_b^k]}(t) g\left( \xsym{\mathcal{Y}}\left(\xvec{x}, t, \frac{t-t_a^k}{t_b^k-t_a^k}\right), \frac{t-t_a^k}{t_b^k-t_a^k} \right)
    \end{equation}
	and then demonstrate that the solution $\theta^{\#}$ to
	\begin{equation}\label{equation:systemctrlnonavcorf}
		\partial_t \theta^{\#} + (\overline{\xvec{y}} \cdot \xnab) \theta^{\#} = f, \quad \theta^{\#}(\cdot, 0) = 0
	\end{equation}
	satisfies $\|\theta^{\#}(\cdot,1)-\theta_1\|_{m} < \varepsilon$, or $\|\theta^{\#}(\cdot,1)-\theta_1\|_{m-1} < \varepsilon \|\theta_1\|_m$ if $\widetilde{\alpha}_{1}, \dots, \widetilde{\alpha}_{4}$ are chosen to depend continuously on $\theta_1$ from the bounded set~$\xB$.

    \paragraph{Step 2. Checking approximate controllability.}
    Since $\theta$ satisfies \eqref{equation:BoussinesqLinearizedNonlocalized}, and recalling that~$\xmcal{Y}$ is the flow associated with $\overline{\xvec{y}}$ from \Cref{theorem:convection}, one finds
    \begin{equation}\label{equation:af1}
    	\theta(\xvec{x}, 1) = \int_0^1 g(\xsym{\mathcal{Y}}(\xvec{x}, 0, r), r) \, \xdx{r}.
    \end{equation}
    Thus, in view of $f$'s definition in \eqref{equation:eqf}, the properties of~$\chi$ and~$\xsym{\mathcal{Y}}$ (\cf~\eqref{equation:Definition_chi}, \eqref{equation:partitionofunity}, and \Cref{theorem:convection}) imply
    \begin{multline}\label{equation:thotolf}
    	\theta(\xvec{x}, 1) =  \sum_{k=1}^K \int_0^1 \chi\left( \xsym{\mathcal{Y}}\left(\xvec{x}, 0, r (t_b^k-t_a^k) + t_a^k \right)\right) g (\xsym{\mathcal{Y}}(\xvec{x}, 0 , r), r) \, \xdx{r} \\
    	\begin{aligned}
    		& = \sum_{k=1}^K \frac{1}{t_b^k-t_a^k} \int_0^1 \mathbb{I}_{[t_a^k, t_b^k]}(s) \chi(\xsym{\mathcal{Y}}(\xvec{x}, 0, s)) g \left(\xsym{\mathcal{Y}}\left(\xvec{x}, 0 , \frac{s-t_a^k}{t_b^k-t_a^k} \right), \frac{s-t_a^k}{t_b^k-t_a^k}\right) \, \xdx{s}\\
    		& = \int_0^1 f(\xsym{\mathcal{Y}}(\xvec{x}, 0, s), s) \, \xdx{s},
    	\end{aligned}
    \end{multline}
    where we used for $k \in \{1,\dots,K\}$ the substitutions $r = (s-t_a^k)(t_b^k-t_a^k)^{-1}$. Therefore, the unique solution $\theta^{\#}$ to the problem \eqref{equation:systemctrlnonavcorf} satisfies $\theta^{\#}(\cdot, 1) = \theta(\cdot, 1)$.

  	\paragraph{Step 3. Average corrections.}
  	In view of the equation for $\theta^{\#}$ in \eqref{equation:systemctrlnonavcorf}, and by employing \cref{equation:gtimespaceaverage,equation:af1,equation:thotolf}	together with the fact that $\xmcal{Y}$  is volume preserving, it follows that
  	\[
  		\int_{\mathbb{T}^2} \theta^{\#}(\xvec{z}, t) \, \xdx{\xvec{z}} = \int_0^t \int_{\mathbb{T}^2} f(\xvec{z}, s) \, \xdx{\xvec{z}}  \xdx{s}, \quad \int_0^1 \int_{\mathbb{T}^2} f(\xvec{z}, s) \, \xdx{ \xvec{z}} \, \xdx{s} = 0,
  	\]
  	where $f$ is the function from \eqref{equation:eqf}. Now, we define
  	\[
  		\Theta(x_1, x_2, t) \coloneq  \theta^{\#}(x_1, x_2, t) - \frac{\chi(x_2) \int_0^t \int_{\mathbb{T}^2} f(\xvec{z}, s) \, \xdx{\xvec{z}} \xdx{s}}{\int_{\mathbb{T}^2} \chi(\xvec{z}) \, \xdx{\xvec{z}}}.
  	\]
  	In particular, we have $\Theta(\cdot, 0) = \theta^{\#}(\cdot, 0)$ and $\Theta(\cdot, 1) = \theta^{\#}(\cdot, 1)$, and~$\Theta$ satisfies \eqref{equation:BoussinesqLinearizedlocalized} with the control
  	\begin{equation}\label{equation:definition_lin_eta}
  		\eta(x_1, x_2, t) \coloneq f(x_1, x_2, t) - \frac{\overline{y}_2(t) \chi'(x_2) \int_0^t \int_{\mathbb{T}^2} f(\xvec{z}, s) \, \xdx{\xvec{z}}  \xdx{s} - \chi(x_2) \int_{\mathbb{T}^2} f(\xvec{z}, t) \, \xdx{\xvec{z}}}{\int_{\mathbb{T}^2} \chi(\xvec{z}) \, \xdx{\xvec{z}}}
  	\end{equation}
  	of the form \eqref{equation:avrg}. Moreover,~$\Theta$ satisfies the controllability condition~\eqref{equation:LinApproxEst}, or $\|\Theta(\cdot,1)-\theta_1\|_{m-1} < \varepsilon \|\theta_1\|_m$ if $\widetilde{\alpha}_1, \dots, \widetilde{\alpha}_4$ are chosen in continuous dependence on~$\theta_1$ from the bounded set $\xB$. 
  	
	\paragraph{Step 4. Structure of the control.}
	It remains to name the profiles~$\zeta_1, \dots, \zeta_{6}$ that where implicitly described during the preceding steps. To this end, the function~$f$ from~\eqref{equation:eqf} is expressed by means of
	 \begin{equation*}
		\begin{aligned}
			f(\xvec{x}, t) & = \chi(x_2) \sum_{k=1}^K  \frac{1}{t_b^k-t_a^k} \mathbb{I}_{[t_a^k, t_b^k]}(t) g\left( \xsym{\mathcal{Y}}\left(\xvec{x}, t, \frac{t-t_a^k}{t_b^k-t_a^k}\right), \frac{t-t_a^k}{t_b^k-t_a^k} \right) \\
			& \, = \sum_{k=1}^K \frac{\chi(x_2)\mathbb{I}_{[t_a^k, t_b^k]}(t)}{t_b^k-t_a^k} g\left(\xsym{\mathcal{Y}}\left(\xvec{x}, t, \sigma(t) \right), \sigma(t) \right),
		\end{aligned}
	\end{equation*}
	where
	\begin{equation}\label{equation:eqfsig}
		\sigma(t) \coloneq \sum_{l=1}^K \mathbb{I}_{[t_a^l, t_b^l]}(t) \frac{t-t_a^l}{t_b^l-t_a^l}.
	\end{equation}
	Finally, after recalling the definitions of $\widetilde{\zeta}_1, \dots, \widetilde{\zeta}_{4}$ in \eqref{equation:T} and of $g$ in Step 1 above, we take the profiles $\zeta_1, \dots, \zeta_{6}$ as the distinct elements of the set
	\begin{equation}\label{equation:ctrlstrcaux}
		\left\{ \chi, \chi' \right\} \cup \left\{(\xvec{x}, t) \mapsto \chi(x_2) \widehat{\zeta}_i(\xsym{\mathcal{U}}(\xsym{\mathcal{Y}}(\xvec{x}, t, 1),1,\sigma(t))) \, \left| \right. \, i \in \{1, \dots, 4\} \right\}.
	\end{equation}
	The order of enumerating this set is irrelevant, but we choose to assign the first two profiles as $\zeta_1 = \chi$ and $\zeta_2 = \chi'$ to match the notations in \Cref{subsection:deg}; then, for the sake of making a definite choice, let us take
	\begin{gather*}
		\zeta_{i}(\xvec{x},t) = \chi(x_2) \widehat{\zeta}_{i-2}(\xsym{\mathcal{U}}(\xsym{\mathcal{Y}}(\xvec{x}, t, 1),1,\sigma(t))), \quad i \in \{3,4,5,6\}.
	\end{gather*}
	The parameters $\alpha_1, \dots, \alpha_{6} \in \xLtwo((0,1); \mathbb{R})$ are subsequently determined from the above choice of $\widetilde{\alpha}_1, \dots, \widetilde{\alpha}_4$ in Step 1, and from the definition of~$\eta$ in \eqref{equation:definition_lin_eta}; namely,
	\begin{gather*}
		\alpha_1(t) = \frac{\int_{\mathbb{T}^2} f(\xvec{z}, t) \, \xdx{\xvec{z}}}{\int_{\mathbb{T}^2} \chi(\xvec{z}) \, \xdx{\xvec{z}}}, \quad \alpha_2(t) = -\frac{\overline{y}_2(t)\int_0^t \int_{\mathbb{T}^2} f(\xvec{z}, s) \, \xdx{\xvec{z}}  \xdx{s}}{\int_{\mathbb{T}^2} \chi(\xvec{z}) \, \xdx{\xvec{z}}}, \\
		\alpha_i(t) = \sum_{k=1}^K \frac{\mathbb{I}_{[t_a^k, t_b^k]}(t)}{t_b^k-t_a^k}\widetilde{\alpha}_{i-2}(\sigma(t)), \quad i \in \{3,4,5,6\}.
	\end{gather*}
	If $\widetilde{\alpha}_{1}, \dots, \widetilde{\alpha}_{4}$ are chosen to depend continuously on $\theta_1$ from the bounded set~$\xB$, the formulas \eqref{equation:eqf} and \eqref{equation:definition_lin_eta} allow taking $\alpha_1, \dots, \alpha_6$ as the image of~$\theta_1$ under a bounded linear operator. Because~$\chi$, $\xsym{\mathcal{Y}}$, $\xsym{\mathcal{U}}$, and~$\sigma$ are universal objects for fixed~$\omegaup$, the set in \eqref{equation:ctrlstrcaux} is not affected by the choice of the initial and target states, the viscosity, the thermal diffusivity, the external forces, and also not by the approximation error specified in Theorems~\Rref{theorem:main} and~\Rref{theorem:secondmain}.

	\end{proof}

	\section{Proofs of the main results}\label{section:proofmainresult}
	In this section, the proofs of Theorems~\Rref{theorem:main} and~\Rref{theorem:secondmain} are presented. To avoid the discussion of weak notions of solutions, we assume without loss of generality that $r \geq 2$ in Theorems~\Rref{theorem:main} and~\Rref{theorem:secondmain}. To see that this is possible, let us assume $0 \leq r < 2$. Then, since $(\xsym{\Phi}_{\operatorname{ext}}, \psi_{\operatorname{ext}}) \in \xLtwo((0, T); \bHn{{2}} \times \xHn{{2}})$, the corresponding uncontrolled weak solution naturally regularizes due to known parabolic smoothing effects, and after any short time assumes a state in $\bHn{{2}}\times\xHn{{2}}$ which can be taken as the new initial data. 
	
	Owing to classical elliptic regularity estimates (\cf~\eqref{equation:velvortest} below), for proving Theorems~\Rref{theorem:main} and~\Rref{theorem:secondmain} it suffices to~determine control forces~$\eta, \overline{\eta} \in \xLtwo((0,T); \xCn{{\infty}}(\mathbb{T}^2;\mathbb{R}))$ of the type \eqref{equation:degeneratecontrol} that ensure an estimate of the form
	\begin{equation}\label{equation:targetconditions_reduced}
		\|\xwcurl{\xvec{u}}(\cdot, T) - \xwcurl{\xvec{u}_T}\|_{r-1} + \|\theta(\cdot, T) - \theta_T\|_{r} + \left| \int_{\mathbb{T}^2} \xvec{u}(\xvec{x}, T) \, \xdx{\xvec{x}} \right| < \varepsilon,
	\end{equation}
	where $\xwcurl{\xvec{u}} \coloneq \partial_1 u_2 - \partial_2 u_1$ is the curl of $\xvec{u}$ and~$\varepsilon > 0$ is the approximation accuracy selected in Theorems~\Rref{theorem:main} and~\Rref{theorem:secondmain}. 
	
	\paragraph{Vorticity-temperature formulation.} For initial data $(\xvec{u}_0, \theta_0) \in \bHn{{r}} \times \xHn{{r}}$ and prescribed forces $(\xsym{\Phi}_{\operatorname{ext}}, \psi_{\operatorname{ext}}) \in \xLtwo((0, T); \bHn{r} \times \xHn{r})$, let $(\xvec{u}, \theta, p)$ be the solution to~\eqref{equation:BoussinesqVelocity}, and denote the vorticity $w = \xwcurl{\xvec{u}}$. Then, the triple~$(\xvec{u}, w, \theta)$ satisfies in~$\mathbb{T}^2 \times (0,T)$ the problem
	\begin{equation}\label{equation:BoussinesqVorticity}
		\begin{gathered}
			\partial_t w - \nu \Delta w + \left(\xvec{u} \cdot \xnab\right) w = \partial_1 \theta + \varphi_{\operatorname{ext}}, \quad
			\partial_t \theta - \tau \Delta \theta + (\xvec{u} \cdot \xnab) \theta = \mathbb{I}_{\omegaup} \eta + \psi_{\operatorname{ext}},\\
			\xwcurl{\xvec{u}} = w, \quad \xdiv{\xvec{u}} = 0, \quad
			w(\cdot, 0) = w_0, \quad \theta(\cdot, 0) = \theta_0,
		\end{gathered}
	\end{equation}
	where $w_0 = \xwcurl{\xvec{u}_0}$ and $\varphi_{\operatorname{ext}} = \xwcurl{\xsym{\Phi}_{\operatorname{ext}}}$.
	Vice versa, if $(\xvec{u}, w, \theta)$ solves \eqref{equation:BoussinesqVorticity} and satisfies $\smallint_{\mathbb{T}^2} u_2(\xvec{x},t) \, \xdx{\xvec{x}} = \smallint_0^t\smallint_{\mathbb{T}^2} \theta(\xvec{x}, s) \, \xdx{\xvec{x}} \xdx{s}$ for all~$t \in [0,T]$, one can recover the pressure~$p$, uniquely up to an additive constant depending on time, such that~$(\xvec{u}, \theta, p)$ is a solution to \eqref{equation:BoussinesqVelocity}; \eg, see \cite{Temam2001}.
	
	\paragraph{Inverting the curl operator.}
	Given any $m \in \mathbb{N}$, let $\xsym{\Upsilon}\colon \xHn{{m-1}}\times\mathbb{R}^2 \longrightarrow \bVn{{m}}$ be the following solenoidal realization of~$(\xwcurl{})^{-1}$: for elements~$z \in \xHn{{m-1}}$ and~$\xvec{A} \in \mathbb{R}^2$, the vector field~$\xsym{\Upsilon}(z,\xvec{A}) \in \bVn{{m}}$ is defined as the unique solution to the planar div-curl problem
	\begin{equation}\label{equation:divcurlprob}
		\xwcurl{\xsym{\Upsilon}(z,\xvec{A})} = z, \quad \xdiv{\xsym{\Upsilon}(z,\xvec{A})} = 0, \quad \int_{\mathbb{T}^2} \xsym{\Upsilon}(z,\xvec{A})(\xvec{x}) \, \xdx{\xvec{x}} = \xvec{A}.
	\end{equation}
	In fact, one has the representation
	\[
		\xsym{\Upsilon}(z,\xvec{A}) = \xnab^{\perp} \psi + \xvec{A}, \quad \xnab^{\perp} \psi \coloneq (\partial_2 \psi, - \partial_1 \psi),
	\]
	where the stream function $\psi$ solves Poisson's equation $\Delta \psi = - z$. Then, by the elliptic theory for the Laplacian, there exists a constant $C_0 > 0$ such that
	\begin{equation}\label{equation:velvortest}
		\|\xsym{\Upsilon}(z,\xvec{A})\|_{m} \leq C_0(\|z\|_{m-1} + |\xvec{A}|)
	\end{equation}
	for all $z \in \xHn{{m-1}}$ and $\xvec{A} \in \mathbb{R}^2$.
	
	\subsection{Well-posedness and hydrodynamic scaling}\label{subsection:wellposedness}
	By analysis similar to that for the Navier--Stokes system in $2$D, the two-dimensional Boussinesq system is globally well-posed in the space
	\begin{gather*}
		\mathcal{X}_{T}^m \coloneq  \mathcal{A}_{T}^{m-1} \times \mathcal{A}_{T}^m, \quad \|(f,g)\|_{\mathcal{X}_{T}^m} \coloneq \|f\|_{\mathcal{A}_{T}^{m-1}}  + \|g\|_{\mathcal{A}_{T}^m},
	\end{gather*}
	where $m \in \mathbb{N}$ and $\mathcal{A}_{T}^m \coloneq \xCzero([0,T];\xHn{{m}}(\mathbb{T}^2; \mathbb{R}))\cap\xLtwo((0,T);\xHn{{m+1}}(\mathbb{T}^2; \mathbb{R}))$ is endowed with $\|\cdot\|_{\mathcal{A}_{T}^m} \coloneq \|\cdot\|_{\xCzero([0,T];\xHn{{m}}(\mathbb{T}^2; \mathbb{R}))} + \|\cdot\|_{\xLtwo((0,T);\xHn{{m+1}}(\mathbb{T}^2; \mathbb{R}))}$.
	
	The following well-posedness result can be shown by analysis similar to the incompressible Navier--Stokes system; \eg, see \cite{Temam2001,FoiasManleyTemam1987}.
	\begin{prpstn}\label{proposition:Wellposedness}
		Given any initial state $(w_0, \theta_0) \in \xHn{{m-1}}\times\xHn{{m}}(\mathbb{T}^2; \mathbb{R})$, external forces~$(h_1, h_2) \in \xLtwo((0,T);\xHn{{m-2}}\times\xHn{{m-1}}(\mathbb{T}^2; \mathbb{R}))$, and average $\xvec{A} \in \xWn{{1,2}}((0,T); \mathbb{R}^2)$, there exists a unique solution~$(w, \theta) \in \mathcal{X}_{T}^m$
		to the Boussinesq system in vorticity-temperature form 
		\begin{equation}\label{equation:GeneralBoussCurl}
			\begin{gathered}
				\partial_t w - \nu \Delta w + \left(\xvec{u} \cdot \xnab\right) w = \partial_1 \theta + h_1, \quad \partial_t \theta - \tau\Delta \theta + \left( \xvec{u} \cdot \xnab\right) \theta = h_2, \\
				\xvec{u}(\cdot, t) = \xsym{\Upsilon}\left(w, \xvec{A} \right), \quad w(\cdot, 0) = w_0, \quad \theta(\cdot, 0) = \theta_0.
			\end{gathered}
		\end{equation}
		The resolving operator ${S}_{T}$ associated with \eqref{equation:GeneralBoussCurl} is the continuous mapping
		\begin{gather*}
			\xHn{{m-1}}\times\xHn{{m}}(\mathbb{T}^2; \mathbb{R}) \times \xLtwo((0,T);\xHn{{m-2}}\times\xHn{{m-1}}(\mathbb{T}^2; \mathbb{R})) \times \xWn{{1,2}}((0,T); \mathbb{R}^2) \longrightarrow \mathcal{X}_{T}^m, \\
			(w_0,\theta_0, h_1, h_2, \xvec{A}) \longmapsto {S}_{T}(w_0,\theta_0, h_1, h_2, \xvec{A}) \coloneq (w, \theta).
		\end{gather*}
	\end{prpstn}

	The next result relates the solutions to \eqref{equation:BoussinesqVorticity} at a small time with the solutions to linear transport
	problems with drift $\overline{\xvec{y}}$ at time $t = 1$. This approach originates from~\cite{Coron96,Coron1996EulerEq} and the present version particularly builds on the recent works~\cite{Nersesyan2021,NersesyanRissel2024}. 
	
	\begin{thrm}\label{theorem:SmallTimeConvergenceToLinearized}
		Given $T > 0$, $m \geq 2$, let the initial state $(w_0, \theta_0) \in \xHn{{m}} \times \xHn{{m+1}}$ and external forces~$(\varphi_{\operatorname{ext}},\psi_{\operatorname{ext}}) \in \xLtwo((0,T); \xHn{{m-2}} \times \xHn{{m-1}})$ be arbitrarily fixed. Moreover, denote by $(v_{\delta}, \vartheta_{\delta})_{\delta \in (0,1)}$ the solution family to the linear transport problems
		\begin{equation}\label{equation:InitStLinearizedlocalized}
			\begin{gathered}
				\partial_t v_{\delta} + (\overline{\xvec{y}} \cdot \xnab) v_{\delta} = \partial_1 \vartheta_{\delta}, \quad
				\partial_t \vartheta_{\delta} + (\overline{\xvec{y}} \cdot \xnab) \vartheta_{\delta} = \eta_{\delta}, \\
				\quad v_{\delta}(\cdot, 0) = w_0, \quad \vartheta_{\delta}(\cdot, 0) = \delta\theta_0,
			\end{gathered}
		\end{equation}
		where $\overline{\xvec{y}}$ is the vector field from \Cref{theorem:convection} and $\eta_{\delta} \subset \xLtwo((0,1); \xCinfty(\mathbb{T}^2; \mathbb{R}))$ is chosen in a way that
		\begin{equation}\label{equation:assumpfrcstlc}
			\begin{gathered}
				\int_{\mathbb{T}^2} \eta_{\delta}(\xvec{x}, \cdot) \, \xdx{\xvec{x}} = 0 \, \mbox{ a.e.}, \quad
				\sup_{t \in [0, 1]} \|\vartheta_{\delta}(\cdot, t)\|_{m+1} = \mathscr{O}(\delta) \mbox{ as } \delta \longrightarrow 0.
			\end{gathered}
		\end{equation}
		Moreover, denote 
		\begin{equation*}
			H_{\delta}(\cdot, t) \coloneq \delta^{-2}\eta_{\delta}(\cdot, \delta^{-1}t), \quad \overline{\xvec{y}}_{\delta}(t) \coloneq \delta^{-1}\overline{\xvec{y}}(\delta^{-1}t).
		\end{equation*}
		Then, one has the convergence
		\begin{equation*}
			\lim\limits_{\delta \to 0}\|S_{\delta}(w_0, \theta_0, \varphi_{\operatorname{ext}}, \psi_{\operatorname{ext}} + H_{\delta}, \overline{\xvec{y}}_{\delta})|_{t=\delta} - (v_{\delta}, \delta^{-1}\vartheta_{\delta})(\cdot,1)\|_{\xHn{{m-1}}\times\xHn{{m}}(\mathbb{T}^2;\mathbb{R})} = 0,
		\end{equation*}
		uniformly with respect to $(\varphi_{\operatorname{ext}},\psi_{\operatorname{ext}})$ from bounded subsets of~$\xLtwo((0,T); \xHn{{m-2}} \times \xHn{{m-1}})$.
	\end{thrm}
	
	\begin{proof}
		For any $\delta \in (0, 1)$, let $(w, \theta) \in \mathcal{X}_{\delta}^m$ be the solution to the nonlinear problem \eqref{equation:BoussinesqVorticity} driven by $H_{\delta}$, and with velocity average $\overline{\xvec{y}}_{\delta}$. Namely, we take $(w, \theta) = S_{\delta}(w_0, \theta_0,  \varphi_{\operatorname{ext}},  \psi_{\operatorname{ext}} + H_{\delta}, \overline{\xvec{y}}_{\delta})$,
		with associated velocity $\xvec{u} = \xsym{\Upsilon}\left(w, \overline{\xvec{y}}_{\delta}\right)$,	where the div-curl solution operator $\xsym{\Upsilon}$ is defined via~\eqref{equation:divcurlprob}.
		Then, we make an ansatz of the form
		\begin{equation}\label{equation:ansatz}
			\begin{gathered}
				w = z_{\delta} + r, \quad \xvec{u} = \overline{\xvec{y}}_{\delta} + \xvec{Z}_{\delta} + \xvec{R}, \quad	\theta = \theta_{\delta} + s,
			\end{gathered}
		\end{equation}
		where
		\[
		z_{\delta}(\cdot,t) \coloneq v_{\delta}(\cdot, \delta^{-1}t), \quad
		\theta_{\delta}(\cdot,t) \coloneq \delta^{-1}\vartheta_{\delta}(\cdot,\delta^{-1}t), \quad \xvec{Z}_{\delta} \coloneq \xsym{\Upsilon}(z_{\delta}, \xsym{0}), \quad \xvec{R} \coloneq \xsym{\Upsilon}(r, \xsym{0}).
		\]
		The theorem will be proved by showing that
		\begin{equation}\label{equation:enestsg}
			\|r(\cdot, \delta)\|_{m-1} + \|s(\cdot, \delta)\|_{m} \longrightarrow 0 \mbox{ as } \delta \longrightarrow 0,
		\end{equation}
		uniformly for $(\varphi_{\operatorname{ext}},\psi_{\operatorname{ext}})$ from bounded subsets of~$\xLtwo((0,1); \xHn{{m-2}} \times \xHn{{m-1}})$.
		\paragraph{Step 1. Description of remainders.}
		By plugging the ansatz \eqref{equation:ansatz} into the equation~\eqref{equation:BoussinesqVorticity} satisfied by $(w, \theta)$, one finds that $r$ and $s$ solve the evolutionary system
		\begin{equation}\label{equation:rem}
			\begin{gathered}
				\partial_t r - \nu \Delta r + \left((\overline{\xvec{y}}_{\delta} + \xvec{Z}_{\delta} + \xvec{R}) \cdot \xnab\right) r + (\xvec{R} \cdot \xnab) z_{\delta} = \Xi_{\delta} + \partial_1 s, \\
				\partial_t s - \tau\Delta s + \left((\overline{\xvec{y}}_{\delta} + \xvec{Z}_{\delta} + \xvec{R}) \cdot \xnab\right) s + (\xvec{R} \cdot \xnab) \theta_{\delta} = \Lambda_{\delta}
			\end{gathered}
		\end{equation}
		with initial conditions $r(\cdot, 0) = s(\cdot, 0) = 0$ and forcing terms
		\begin{equation*}
			\begin{gathered}
				\Xi_{\delta} \coloneq \varphi_{\operatorname{ext}} - (\xvec{Z}_{\delta} \cdot \xnab) z_{\delta} + \nu \Delta z_{\delta}, \quad \Lambda_{\delta} \coloneq \psi_{\operatorname{ext}} - (\xvec{Z}_{\delta} \cdot \xnab) \theta_{\delta} + \tau \Delta \theta_{\delta}.
			\end{gathered}
		\end{equation*}
		Moreover, one has the elliptic estimates (\cf~\eqref{equation:velvortest})
		\begin{equation}\label{equation:ellestRZ}
			\|\xvec{R}(\cdot, t)\|_{m} \leq C_0\|r(\cdot, t)\|_{m-1}, \quad \|\xvec{Z}_{\delta}(\cdot, t)\|_{m} \leq C_0 \|z_{\delta}(\cdot, t)\|_{m-1}, \quad t \in [0, \delta].
		\end{equation}
		\paragraph{Step 2. \Apriori~estimates.}
		Since the vorticity-temperature coupling in \eqref{equation:BoussinesqVorticity} is linear, and due to the dissipation of both $w$ and $\theta$, the subsequent estimates are similar to those provided for the Navier--Stokes system in \cite[Proof of Lemma 5.5]{NersesyanRissel2024} and \cite{Nersesyan2021}*{Proposition 2.2}. First, we formally multiply the equations in \eqref{equation:rem} with~$(-\Delta)^{m-1} r$ and $(-\Delta)^{m} s$ respectively. Then we use integration by parts, Poincar\'e's inequality, \eqref{equation:ellestRZ}, and the periodic boundary conditions. As a result, 
		{\allowdisplaybreaks
			\begin{multline*}
				\frac{1}{2}\|r(\cdot, t)\|_{m-1}^2 + \frac{1}{2}\|s(\cdot, t)\|_{m}^2 + \nu \int_0^t \| r(\cdot, \sigma) \|_{m}^2 \, \xdx{\sigma} + \tau\int_0^t \| s(\cdot, \sigma) \|_{m+1}^2 \, \xdx{\sigma} \\
				\begin{aligned}
					& \leq \int_0^t \left(\|\Xi_{\delta}(\cdot, \sigma)\|_{m-2} \|r(\cdot, \sigma)\|_{m} \, \xdx{\sigma} + \|\Lambda_{\delta}(\cdot, \sigma)\|_{m-1} \|s(\cdot, \sigma)\|_{m+1}\right) \, \xdx{\sigma} \\
					& \quad \, + \int_0^t \|\xvec{R}(\cdot, \sigma)\|_{m} \left( 	\|z_{\delta}(\cdot, \sigma)\|_{m}\|r(\cdot, \sigma)\|_{m-1} + \|r(\cdot, \sigma)\|_{m-1}\|r(\cdot, \sigma)\|_{m} \right) \, \xdx{\sigma}\\
					& \quad \, + \int_0^t \|\xvec{R}(\cdot, \sigma)\|_{m} \left( 	\|\theta_{\delta}(\cdot, \sigma)\|_{m+1}\|s(\cdot, \sigma)\|_{m} + \|s(\cdot, \sigma)\|_{m}\|s(\cdot, \sigma)\|_{m+1} \right) \, \xdx{\sigma}\\
					& \quad \, + \int_0^t \left(\|\overline{\xvec{y}}_{\delta}(\sigma) + 	\xvec{Z}_{\delta}(\cdot, \sigma)\|_{m+1} \|r(\cdot, \sigma)\|_{m-1}^2 + \|s(\cdot, \sigma)\|_{m} \|r(\cdot, \sigma)\|_{m-1} \right)\, \xdx{\sigma} \\
					& \quad \, + \int_0^t \|\overline{\xvec{y}}_{\delta}(\sigma) + 	\xvec{Z}_{\delta}(\cdot, \sigma)\|_{m+1} \|s(\cdot, \sigma)\|_{m}^2 \, \xdx{\sigma}.
				\end{aligned}
			\end{multline*}
		}\noindent
		In order to further estimate the right-hand side of the previous inequality, we again use~\eqref{equation:ellestRZ}. Hereby, we also account for the $\delta$-scaling by substituting~$\sigma \leftrightarrow \delta \sigma$ under several of the integral signs. \Eg, it follows that
		\begin{equation}\label{equation:gfe}
			\begin{gathered}
				\int_0^t \|f(\cdot, \sigma)\|_{l} \, \xdx{\sigma} \leq \min \left\{ \delta \int_0^1 \|f(\cdot, \delta \sigma)\|_{l} \, \xdx{\sigma}, \,  \int_0^\delta \|f(\cdot, \sigma)\|_l \, \xdx{\sigma} \right\}
			\end{gathered}
		\end{equation}
		for~$t \in (0, \delta)$ and $f \in \xLone((0, \delta); \xHn{{l}}(\mathbb{T}^2; \mathbb{R}))$ with~$l \geq 0$. The relations in~\eqref{equation:gfe}, combined with the respective boundedness of~$\Xi_{\delta}$ in~$\xLone((0, \delta); \xHn{{m-2}}(\mathbb{T}^2; \mathbb{R}))$ and of~$\overline{\xvec{y}}$ in $\xCzero([0,1]; \mathbb{R}^2)$, yield $\lim\limits_{\delta \to 0} \smallint_0^{\delta} \|\Xi_{\delta}(\cdot, \sigma)\|_{m-2} \, \xdx{\sigma} = 0$ and
		\begin{gather*}
			\lim\limits_{\delta \to 0} \int_0^{\delta} \|\overline{\xvec{y}}_{\delta}(\sigma) + 	\xvec{Z}_{\delta}(\cdot, \sigma)\|_{m+1} \, \xdx{\sigma}  \leq \sup_{s \in [0, 1]} |\overline{\xvec{y}}(s)|.
		\end{gather*}
		In particular, thanks to the assumptions in~\eqref{equation:assumpfrcstlc}, one can infer 
		\begin{gather*}
			\int_0^{\delta} \|\theta_{\delta}(\cdot, \sigma)\|_{m+1}^2 \, \xdx{\sigma} \leq \delta^{-1}\sup_{s \in [0,1]} \|\vartheta_{\delta}(\cdot, s)\|_{m+1}^2 = \mathscr{O}(\delta) \mbox{ as } \delta \longrightarrow 0, \\
			\lim\limits_{\delta \to 0} \int_0^{\delta} \|\Lambda_{\delta}(\cdot, \sigma)\|_{m-1} \, \xdx{\sigma} = 0.
		\end{gather*}
		
		Therefore, in view of Gr\"onwall's lemma, and by essentially copying the analysis from \cite[Proof of Lemma 5.5]{NersesyanRissel2024}, it follows that there is a constant $C > 0$, which is independent of~$\delta \in (0,1)$, $t \in [0,\delta]$, and~$(w_0, \theta_0)$ varying in a bounded subset of~$\xHn{{m}}\times\xHn{{m+1}}$, such that  
		\begin{equation*}
			\|r(\cdot, t)\|_{m-1}^2 + \|s(\cdot, t)\|_{m}^2 \leq C_{\delta} + C\int_0^t \left(\|r(\cdot, \sigma)\|_{m-1}^4 + \|s(\cdot, \sigma)\|_{m}^4\right) \, \xdx{\sigma},
		\end{equation*}
		where the family of constants $(C_{\delta})_{\delta\in(0,1)}$ satisfies $\lim_{\delta \to 0} C_{\delta} = 0$. Finally, after denoting
		\[
		\Psi(t) \coloneq C_{\delta} + C\int_0^t \left(\|r(\cdot, \sigma)\|_{m-1}^4 + \|s(\cdot, \sigma)\|_{m}^4\right) \, \xdx{\sigma},
		\]
		the convergence asserted in \eqref{equation:enestsg} follows by utilizing that $\Psi$ meets the differential inequality $\xdrv{}{t}\Psi \leq C \Psi^2$; for situations of similar nature, see, \eg,~\cite{Nersesyan2021}*{Proposition 2.2} and \cite[Proof of Lemma 5.5]{NersesyanRissel2024}.
	\end{proof}

	\subsection{Controllability of the vorticity-temperature formulation}\label{subsection:vortcontrlglb}
	We denote by $\zeta_1, \dots, \zeta_{6} \in \xLtwo((0,1); \xCinfty(\mathbb{T}^2; \mathbb{R}))$ the profiles that appear in \Cref{subsection:deg} and which are constructed in \Cref{theorem:localized}. Moreover, we recall that $\overline{\xvec{y}}$ is obtained via \Cref{theorem:convection} and use the short notation $\overline{\xvec{y}}_{\delta}(t) = \delta^{-1}\overline{\xvec{y}}(\delta^{-1}t)$ for $\delta \in (0,1)$ and $t \in [0, \delta]$.
	
	The following result is a consequence of Theorems~\Rref{theorem:localized} and \Rref{theorem:SmallTimeConvergenceToLinearized}. It provides physically localized controls for steering the temperature towards any admissible target in a small time, while keeping the vorticity close to the initial one. 
	\begin{thrm}\label{theorem:main_smalltimelargecontrol}
		For any given ${T} > 0$, $m \geq 2$, $(w_0, \theta_0, \theta_1) \in \xHn{{m}} \times \xHn{{m+1}} \times \xHn{{m+1}}$, and~$(\varphi_{\operatorname{ext}}, \psi_{\operatorname{ext}}) \in \xLtwo((0, {T}); \xHn{{m-2}} \times \xHn{{m-1}})$, there exist parameters
		\begin{equation}\label{equation:prevy}
			(\alpha_{\delta,1}, \dots, \alpha_{\delta, 6})_{\delta \in (0,1)} \subset \xLtwo((0,1); \mathbb{R}), \quad \sum_{l=1}^{6} \int_{\mathbb{T}^2} \alpha_{\delta, l} \zeta_l (\xvec{x}, \cdot) \, \xdx{\xvec{x}} = 0 \, \mbox{ a.e.}
		\end{equation}
		such that, when~$\delta \longrightarrow 0$, one has in~$\xHn{{m-1}} \times \xHn{{m}}$ the convergence
		\begin{equation*}
			{S}_{\delta}\left(w_0, \theta_0, \varphi_{\operatorname{ext}}, \psi_{\operatorname{ext}} + \delta^{-2}\sum_{l=1}^{6} \alpha_{\delta, l}(\delta^{-1}\cdot) \zeta_l(\cdot, \delta^{-1}\cdot), \overline{\xvec{y}}_{\delta}\right)\!|_{t=\delta} \longrightarrow (w_0, \theta_1),
		\end{equation*}
		uniformly for $(\varphi_{\operatorname{ext}}, \psi_{\operatorname{ext}})$ from bounded subsets of $\xLtwo((0,{T}); \xHn{{m-2}} \times \xHn{{m-1}})$.
	\end{thrm}
	\begin{proof}
		Let $\rho \in (0, 1)$ and take $(\widetilde{v}_{\rho}, \widetilde{\vartheta}_{\rho})$ as the unique solution to the homogeneous linear problem
		\begin{equation}\label{equation:FreeBoussinesqLinearizedlocalized}
			\partial_t \widetilde{v}_{\rho} + (\overline{\xvec{y}} \cdot \xnab) \widetilde{v}_{\rho} = \partial_1 \widetilde{\vartheta}_{\rho}, \quad \partial_t \widetilde{\vartheta}_{\rho} + (\overline{\xvec{y}} \cdot \xnab) \widetilde{\vartheta}_{\rho} = 0, \quad
			(\widetilde{v}_{\rho}, \widetilde{\vartheta}_{\rho})(\cdot, 0) = (w_0, \rho\theta_0).
		\end{equation}
		By the properties of $\xsym{\mathcal{Y}}$ from \Cref{theorem:convection}, it follows that
		\[
			\widetilde{\vartheta}_{\rho}(\cdot, 1) = \rho\theta_0, \quad \widetilde{v}_{\rho}(\cdot, 1) = w_0 + \widetilde{v}_{\rho, 1}, \quad \widetilde{v}_{\rho, 1} \coloneq \int_0^1 \partial_1 \widetilde{\vartheta}_{\rho}(\xsym{\mathcal{Y}}(\cdot, 1, s), s) \, \xdx{s}.
		\]
		In addition, basic estimates yield
		\begin{equation}\label{equation:bet}
			\sup\limits_{t \in [0,1]}\|\widetilde{\vartheta}_{\rho}(\cdot, t)\|_{m+1} + \|\widetilde{v}_{\rho, 1}\|_{m} = \mathscr{O}(\rho) \mbox{ as } \rho \longrightarrow 0.
		\end{equation}
		Now, let any $\varepsilon > 0$ be fixed. Then, for each $\rho \in (0, 1)$, an application of \Cref{theorem:localized} with the target $\Theta_1 \coloneq \rho(\theta_1 - \theta_0)$
		provides $(\alpha_{\rho,1}, \dots, \alpha_{\rho, 6})_{\rho \in (0,1)} \subset \xLtwo((0,1); \mathbb{R})$ such that the respective solution
		\begin{equation*}
			(V_{\rho}, \Theta_{\rho}) \in \xCzero([0,1];\xHn{{m}}\times\xHn{{m+1}})\cap\xWn{{1,2}}((0,1);\xHn{{m-1}}\times\xHn{{m}})
		\end{equation*}
		to
		\begin{equation}\label{equation:VTH}
			\partial_t V_{\rho} + (\overline{\xvec{y}} \cdot \xnab) V_{\rho} = \partial_1 \Theta_{\rho}, \quad \!\!
			\partial_t \Theta_{\rho} + (\overline{\xvec{y}} \cdot \xnab) \Theta_{\rho} = \sum_{l=1}^{6} \alpha_{\rho, l} \zeta_l, \quad \!\! 
			V_{\rho}(\cdot, 0) = \Theta_{\rho}(\cdot, 0) = 0,
		\end{equation} 
		satisfies
		\[
			\|\Theta_{\rho}(\cdot,1)-\Theta_1\|_{m} < \varepsilon \rho.
		\]
		Moreover, the control parameters are chosen such that
		\begin{equation}\label{equation:corcob}
			\begin{gathered}
				\|\sum_{l=1}^{6} \alpha_{\rho, l} \zeta_l\|_{\xLtwo((0,1);\xHn{{m+1}})} \leq \rho c_{\varepsilon} \left(\sum_{l=1}^{6}\|\zeta_l\|_{\xLtwo((0,1);\xHn{{m+1}})}\right)  \|\theta_1 - \theta_0\|_{m+1},
				\\
				\sum_{l=1}^{6} \int_{\mathbb{T}^2} \alpha_{\rho, l} \zeta_l (\xvec{x}, \cdot) \, \xdx{\xvec{x}} = 0 \, \mbox{ a.e.},
			\end{gathered}
		\end{equation}
		where $c_{\varepsilon}$ is the operator norm of the bounded linear operator $\mathcal{C}_{\varepsilon}$ from \Cref{theorem:localized}.
		Further, by estimating $V_{\rho}$ and $\Theta_{\rho}$ in \eqref{equation:VTH} using \eqref{equation:corcob}, one finds
		\begin{equation}\label{equation:coroe}
			\sup\limits_{t \in [0,1]}\|\Theta_{\rho}(\cdot, t)\|_{m+1} + \|w_0 - \widetilde{v}_{\rho}(\cdot, 1)\|_{m} + \|V_{\rho}(\cdot, 1)\|_{m} = \mathscr{O}(\rho) \mbox{ as } \rho \longrightarrow 0.
		\end{equation}
		Now, the pair $(v_{\rho}, \vartheta_{\rho}) \coloneq (\widetilde{v}_{\rho} + V_{\rho}, \widetilde{\vartheta}_{\rho} + \Theta_{\rho})$ solves the linear reference system \eqref{equation:InitStLinearizedlocalized} in \Cref{theorem:SmallTimeConvergenceToLinearized} with force $\eta_{\rho} = \sum_{l=1}^{6} \alpha_{\rho, l} \zeta_l$, and it holds
		\[
			\|\vartheta_{\rho}(\cdot, 1) - \rho \theta_1\|_{m} < \varepsilon \rho.
		\]
		Moreover, from \eqref{equation:bet} and \eqref{equation:coroe} it follows that
		\begin{equation*}
			\sup_{t \in [0, 1]} \|\vartheta_{\rho}(\cdot, t)\|_{m+1} + \|v_{\rho}(\cdot,1) - w_0\|_{m-1} = \mathscr{O}(\rho) \mbox{ as } \rho \longrightarrow 0.
		\end{equation*}
		Hence, by \Cref{theorem:SmallTimeConvergenceToLinearized}, there exists~$\delta_* > 0$ such that
		\[
			\| S_{\delta}\left(w_0, \theta_0,  \varphi_{\operatorname{ext}},  \psi_{\operatorname{ext}} + \delta^{-2} \eta_{\delta}(\cdot, \delta^{-1}\cdot), \overline{\xvec{y}}_{\delta} \right)\!|_{t=\delta} - (w_0, \theta_1) \|_{\xHn{{m-1}}\times\xHn{{m}}} < \varepsilon
		\]
		for all~$\delta \in (0, \delta_*)$.
	\end{proof}

	The following result specifies types of target states that are approximately reached (without using a control) when starting from appropriate initial data. To this end, let $\Pi_1 \colon\mathcal{A}_{T}^{m} \times \mathcal{A}_{T}^{m+1} \longrightarrow \mathcal{A}_{T}^{m}$ denote the projection to the first component.
	
	\begin{thrm}\label{theorem:propvortcontrl}
		Let~$m \geq 2$, ${T} > 0$, $\widetilde{\xi}, \widehat{\xi} \in \xCinfty(\mathbb{T}^2; \mathbb{R})$ with zero averages, $(w_0, \theta_0) \in \xHn{{m}} \times \xHn{{m+1}}$, and~$(\varphi_{\operatorname{ext}}, \psi_{\operatorname{ext}}) \in \xLtwo((0, {T}); \xHn{{m-2}} \times \xHn{{m-1}})$ be given. As $\delta \longrightarrow 0$, one has
		\begin{gather}\label{equation:propvortcontrl0}
			\Pi_1 S_{\delta}(w_0 , \theta_0 - \delta^{-1} \widetilde{\xi}, \varphi_{\operatorname{ext}}, \psi_{\operatorname{ext}}, \xsym{0}) |_{t=\delta}  \longrightarrow w_0 - \partial_1\widetilde{\xi},\\
			S_{\delta}(w_0 \!+\! \delta^{-1/2}\widehat{\xi}, \theta_0, \varphi_{\operatorname{ext}}, \psi_{\operatorname{ext}}, \xsym{0})\! |_{t=\delta} \!-\! (\delta^{-1/2}\widehat{\xi}, 0) \! \longrightarrow \! (w_0 \!-\! (\xsym{\Upsilon}(\widehat{\xi}, \xsym{0})  \cdot  \xnab) \widehat{\xi}, \theta_0)\label{equation:propvortcontrl1}
		\end{gather}
		in the norms of $\xHn{{m-1}}$ and $\xHn{{m-1}}\times\xHn{m}$ respectively. Both convergences in \eqref{equation:propvortcontrl0} and \eqref{equation:propvortcontrl1} are uniform with respect to $(\varphi_{\operatorname{ext}}, \psi_{\operatorname{ext}})$ from bounded subsets of $\xLtwo((0, {T}); \xHn{{m-2}} \times \xHn{{m-1}})$.
	\end{thrm}
	\begin{proof}
		We develop a simplified version of the arguments given in \cite[Proof of Proposition 1.2]{BoulvardGaoNersesyan2023} for the $3$D primitive equations. 
		First, we observe that
		\[
			(W_{\delta}, \Theta_{\delta}) = S_{\delta}(w_0, \theta_0 - \delta^{-1} \widetilde{\xi}, \varphi_{\operatorname{ext}}, \psi_{\operatorname{ext}}, \xsym{0}) + (0, \delta^{-1} \widetilde{\xi})
		\]
		solves
		\begin{equation}\label{equation:BoussinesqMultiplicativeControl}
			\begin{gathered}
				\partial_t W_{\delta} - \nu \Delta W_{\delta} + \left(\xvec{U}_{\delta} \cdot \xnab\right) W_{\delta} = \partial_1 (\Theta_{\delta} - \delta^{-1} \widetilde{\xi}) + \varphi_{\operatorname{ext}}, \\
				\partial_t \Theta_{\delta} - \tau \Delta (\Theta_{\delta} - \delta^{-1} \widetilde{\xi}) + (\xvec{U}_{\delta} \cdot \xnab) (\Theta_{\delta} - \delta^{-1} \widetilde{\xi}) = \psi_{\operatorname{ext}},\\
				\xwcurl{\xvec{U}_{\delta}} = W_{\delta}, \quad \xdiv{\xvec{U}_{\delta}} = 0, \quad \int_{\mathbb{T}^2} \xvec{U}_{\delta}(\xvec{x}, t) \, \xdx{\xvec{x}} = \xsym{0},  \\
				W_{\delta}(\cdot, 0) = W_0 \coloneq w_0, \quad \Theta_{\delta}(\cdot, 0) = \Theta_0 \coloneq \theta_0.
			\end{gathered}
		\end{equation}
		
		\paragraph{Showing \eqref{equation:propvortcontrl0}.} It remains to verify for the solutions $(W_{\delta}, \xvec{U}_{\delta}, \Theta_{\delta})_{\delta \in (0,1)}$ to \eqref{equation:BoussinesqMultiplicativeControl} the convergence
			\begin{equation}\label{equation:convergencedeltaW}
				\lim\limits_{\delta \to 0} \| W_{\delta}(\cdot, \delta) - (W_0 - \partial_1\widetilde{\xi}) \|_{m-1} = 0.
			\end{equation}
		Establishing \eqref{equation:convergencedeltaW} means showing $\|Q_{\delta}(\cdot, \delta)\|_{m-1} \longrightarrow 0$ as $\delta \longrightarrow 0$ for the functions
		\[
			Q_{\delta}(\xvec{x}, t) \coloneq W_{\delta}(\xvec{x}, t) - W_0(\xvec{x}) + \delta^{-1} t \partial_1 \widetilde{\xi} (\xvec{x}), \quad \delta \in (0,1).
		\]
		In fact, each $Q_{\delta}$ meets the initial condition $Q_{\delta}(\cdot, 0) = 0$ and solves
		\begin{equation}\label{equation:mcrt1}
			\begin{gathered}
				\partial_t Q_{\delta} - \nu \Delta Q_{\delta} + (\xvec{V}_{\delta} \cdot \xnab) Q_{\delta} = \varphi_{\operatorname{ext}} + \partial_1 (\Theta_{\delta} - R_{\delta}) + \partial_1 R_{\delta} - \nu \Delta (Q_{\delta} - W_{\delta}) \\ 
				+ ((\xvec{U}_{\delta} - \xvec{V}_{\delta}) \cdot \xnab) (Q_{\delta} - W_{\delta}) + (\xvec{V}_{\delta} \cdot \xnab) (Q_{\delta} - W_{\delta}) + ((\xvec{V}_{\delta} - \xvec{U}_{\delta}) \cdot \xnab) Q_{\delta},
			\end{gathered}
		\end{equation}
		where $	\xvec{V}_{\delta} = \xsym{\Upsilon}\left(Q_{\delta}, \xsym{0} \right)$
		and
		\begin{equation}\label{equation:Rdelta}
			\begin{aligned}
				R_{\delta}(\cdot, t) \coloneq \Theta_{\delta}(\cdot, t) - \Theta_0 
				+ \delta^{-1}t \tau \Delta\widetilde{\xi} - \delta^{-1}t  \left(\xsym{\Upsilon}\left(W_0 - \frac{\delta^{-1} t \partial_1 \widetilde{\xi}}{2}, \xvec{0}\right) \cdot \xnab \right) \widetilde{\xi}
			\end{aligned}
		\end{equation}
		for $\delta \in (0,1)$ and $t \in [0, \delta]$. In particular, one has~$R_{\delta}(\cdot, 0) = 0$, and~$R_{\delta}$ satisfies
		\begin{equation}\label{equation:mcrt2}
			\begin{gathered}
				\partial_t R_{\delta} - \tau\Delta R_{\delta} + (\xvec{V}_{\delta} \cdot \xnab) R_{\delta} = \psi_{\operatorname{ext}} - \tau\Delta (R_{\delta} - \Theta_{\delta}) + (\xvec{V}_{\delta} \cdot \xnab) (R_{\delta} - \Theta_{\delta})  \\
				\quad \, + ((\xvec{V}_{\delta} - \xvec{U}_{\delta}) \cdot \xnab) R_{\delta} + ((\xvec{V}_{\delta} - \xvec{U}_{\delta}) \cdot \xnab) (\Theta_{\delta} - R_{\delta}) + \delta^{-1} (\xvec{V}_{\delta} \cdot \xnab) \widetilde{\xi}.
			\end{gathered}
		\end{equation}
		Now, given any $a \geq 1$, it follows that
		\begin{equation}\label{equation:asymptQWRT}
			\|(Q_{\delta} - W_{\delta})\|_{\xLn{{a}}((0, \delta); \xHn{{m}})}^{a} + \|(R_{\delta} - \Theta_{\delta})\|_{\xLn{{a}}((0, \delta); \xHn{{m+1}})}^{a} = \mathscr{O}(\delta)
		\end{equation}
		as $\delta \longrightarrow 0$; \eg, for the second term this can be seen via
		\begin{multline*}
			\|(R_{\delta} - \Theta_{\delta})\|_{\xLn{{a}}((0, \delta); \xHn{{m+1}})}^{a} \\
			\begin{aligned}
				& = \int_0^{\delta} \|\Theta_0 - \delta^{-1}s \Delta\widetilde{\xi} + \delta^{-1}s \left(\xsym{\Upsilon}\left(W_0 - \frac{\delta^{-1} s \partial_1 \widetilde{\xi}}{2}, \xvec{0}\right) \cdot \xnab \right) \widetilde{\xi}\|_{m+1}^a \, \xdx{s} \\
				& \leq \delta C \left(1 + \|W_0\|_{m}^{2a} + \|\Theta_0\|_{m+1}^a + \| \widetilde{\xi} \|_{m+4}^{2a}\right) .
			\end{aligned}
		\end{multline*}
		The argument is then completed by utilizing \cref{equation:mcrt1,equation:asymptQWRT,equation:mcrt2} to derive energy estimates for~$R_{\delta}$ and~$Q_{\delta}$.  Indeed, one can begin with formally multiplying~\eqref{equation:mcrt1} and~\eqref{equation:mcrt2} by~$(-\Delta)^{m-1} Q_{\delta}$ and~$(-\Delta)^{m} R_{\delta}$ respectively. Then, integration by parts, Poincar\'e's inequality, Sobolev embeddings, Gr\"onwall's inequality, and \eqref{equation:asymptQWRT} imply that~$\|Q_{\delta}(\cdot, \delta)\|_{m-1} \longrightarrow 0$ as~$\delta \longrightarrow 0$. From the viewpoint of estimates, the procedure is similar to that presented in \enquote{Step 2} of the proof for \Cref{theorem:SmallTimeConvergenceToLinearized}. 
	
		\paragraph{Showing \eqref{equation:propvortcontrl1}.} Instead of \eqref{equation:BoussinesqMultiplicativeControl} we consider the time evolution of the modified trajectory
		\begin{align*}
			(W_{\delta}, \Theta_{\delta}) \coloneq S_{\delta}(w_0 + \delta^{-1/2} \widehat{\xi}, \theta_0, \varphi_{\operatorname{ext}}, \psi_{\operatorname{ext}}, \xsym{0}) - (\delta^{-1/2}\widehat{\xi}, 0),
		\end{align*}
		which solves the problem
		\begin{gather*}
			\begin{aligned}
				\partial_t W_{\delta} - \nu \Delta W_{\delta} + \left(\xvec{U}_{\delta} \cdot \xnab\right) W_{\delta} = \varphi_{\operatorname{ext}} + \partial_1 \Theta_{\delta} + \nu \delta^{-1/2} \Delta \widehat{\xi} \\ - \delta^{-1/2} \left(\xvec{U}_{\delta} \cdot \xnab\right) \widehat{\xi}
				- \delta^{-1/2} \left(\xsym{\Upsilon}(\widehat{\xi}, \xsym{0}) \cdot \xnab \right) W_{\delta} - \delta^{-1} \left(\xsym{\Upsilon}(\widehat{\xi}, \xsym{0}) \cdot \xnab \right) \widehat{\xi},
			\end{aligned}\\
			\xwcurl{\xvec{U}_{\delta}} = W_{\delta}, \quad \xdiv{\xvec{U}_{\delta}} = 0, \quad \int_{\mathbb{T}^2} \xvec{U}_{\delta}(\xvec{x}, t) \, \xdx{\xvec{x}} = \xsym{0},\\
			\partial_t \Theta_{\delta} - \tau \Delta \Theta_{\delta} + (\xvec{U}_{\delta} \cdot \xnab) \Theta_{\delta} = \psi_{\operatorname{ext}} - \delta^{-1/2}\left(\xsym{\Upsilon}(\widehat{\xi}, \xsym{0}) \cdot \xnab \right)\Theta_{\delta},\\
			W_{\delta}(\cdot, 0) = w_0, \quad \Theta_{\delta}(\cdot, 0) = \theta_0.
		\end{gather*}
		Then, the convergence in \eqref{equation:propvortcontrl1} follows after showing that the remainder terms
		\begin{equation*}
			\begin{gathered}
				Q_{\delta}(\cdot, t)
				\coloneq W_{\delta}(\cdot, t) - W_0 + \delta^{-1} t \left(\xsym{\Upsilon}(\widehat{\xi},\xsym{0})\cdot\xnab\right)\widehat{\xi} - \delta^{-1/2} t \nu \Delta \widehat{\xi}, \\ 
				R_{\delta}(\cdot, t) \coloneq \Theta_{\delta}(\cdot, t) - \Theta_0
			\end{gathered}
		\end{equation*}
		satisfy $(Q_{\delta}, R_{\delta})(\cdot, \delta) \longrightarrow (0, 0)$ in $\xHn{{m-1}}\times\xHn{{m}}$ as $\delta \longrightarrow 0$.
		Indeed, instead of \eqref{equation:asymptQWRT}, one now has
		\begin{equation*}
			\|(Q_{\delta} - W_{\delta})\|_{\xLn{{a}}((0, \delta); \xHn{{m}})}^{a} + \|(R_{\delta} - \Theta_{\delta})\|_{\xLn{{a}}((0, \delta); \xHn{{m+1}})}^{a} = \mathscr{O}(\delta^{1/2}),
		\end{equation*}
		which suffices in order to derive standard energy estimates for $Q_{\delta}$ and $R_{\delta}$ as explained in the proof of \Cref{theorem:SmallTimeConvergenceToLinearized}.
	\end{proof}
	
	We also need the below auxiliary result which can be proved by elementary trigonometric calculations.
	\begin{lmm}\label{lemma:trigo}
		Let $\mathcal{E}$ be the collection of $\sin(\xvec{x}\cdot \xvec{n})$ and $\cos(\xvec{x}\cdot \xvec{n})$ with $\xvec{n}\in \mathbb{N}\times(\mathbb{N}\cup\{0\})$. Then, the set
		\[
			\mathcal{H} \coloneq \left\{ \xi_0 + \left(\xsym{\Upsilon}(\xi_1, \xsym{0}) \cdot \xnab\right) \xi_1 + \left(\xsym{\Upsilon}(\xi_2, \xsym{0}) \cdot \xnab\right) \xi_2 \, \, | \, \, \xi_0,\xi_1,\xi_2\in \operatorname{span}_{\mathbb{R}}\mathcal{E} \right\}
		\]
		contains $\pm\sin(\xvec{x}\cdot \xvec{n})$ and $\pm\cos(\xvec{x}\cdot \xvec{n})$ for all nonzero $\xvec{n}\in \mathbb{Z}^2$.
	\end{lmm}
	\begin{proof}
		The idea is to utilize the representations $\xsym{\Upsilon}(\sin(\xvec{n}\cdot \xvec{x}), \xsym{0}) = \xvec{n}^{\perp}|\xvec{n}|^{-2} \cos(\xvec{n}\cdot \xvec{x})$ and $\xsym{\Upsilon}(\cos(\xvec{n}\cdot \xvec{x}), \xsym{0}) = -\xvec{n}^{\perp}|\xvec{n}|^{-2} \sin(\xvec{n}\cdot \xvec{x})$, and then to express $\sin(\xvec{x}\cdot \xvec{n})$ and $\cos(\xvec{x}\cdot \xvec{n})$ via trigonometric angle identities as elements of $\mathcal{H}$. Such arguments are described, \eg, in \cite{AgrachevSarychev2006}. 
	\end{proof}
	
	The following corollary provides the structure of finitely decomposable controls that arise from a certain combination of the results from Theorems~\Rref{theorem:main_smalltimelargecontrol} and \Rref{theorem:propvortcontrl}.
	\begin{crllr}\label{Corollary:st}
		Let~$m \geq 2$, ${T} > 0$, $\xi = \partial_1 \kappa$ for some $\kappa \in \xCinfty(\mathbb{T}^2; \mathbb{R})$, $(w_0, \theta_0) \in \xHn{{m}} \times \xHn{{m+1}}$, and~$(\varphi_{\operatorname{ext}}, \psi_{\operatorname{ext}}) \in \xLtwo((0, {T}); \xHn{{m-2}} \times \xHn{{m-1}})$. There exist control parameters $(\widetilde{\gamma}_{\delta}, \widetilde{\gamma}_{\delta,1}, \dots, \widetilde{\gamma}_{\delta,6})_{\delta > 0} \subset \xLtwo((0, T); \mathbb{R})$ and $\widetilde{\xsym{\aleph}}_{\delta}\in\xCinfty_0((0,T);\mathbb{R})$ such that one has in $\xHn{{m-1}}$ the convergence
	\begin{multline}\label{equation:corcon}
		\Pi_1S_{\delta}\left(w_0, \theta_0, \varphi_{\operatorname{ext}}(T-\delta+\cdot), \psi_{\operatorname{ext}}(T-\delta+\cdot) + \widetilde{\eta}_{\delta}, \widetilde{\xsym{\aleph}_{\delta}}\right) |_{t=\delta} - \delta^{-1/2} \xi \\ \longrightarrow w_0 - (\xsym{\Upsilon}(\xi, \xsym{0}) \cdot \xnab) \xi \mbox{ as }  \delta \longrightarrow 0,
	\end{multline}
	where $\widetilde{\eta}_{\delta}(\xvec{x}, t) \coloneq \sum_{l = 1}^{6} \widetilde{\gamma}_{\delta,l}(t) \zeta_l(\xvec{x}, \widetilde{\gamma}_{\delta}(t))$.
	\end{crllr}
	\begin{proof}
		The argument consists of the following three parts; see \eqref{equation:defsp} for the precise definition of the coefficients in the representation of $\widetilde{\eta}_{\delta}$. 1) Let $\varepsilon > 0$ be arbitrary. By \Cref{theorem:propvortcontrl}, one can take $0 < \delta_3 < T$ so small that the vorticity component of the left-hand side in \eqref{equation:propvortcontrl1}, with $\delta = \delta_3$, $(\varphi_{\operatorname{ext}}, \psi_{\operatorname{ext}})$ replaced by $(\varphi_{\operatorname{ext}}, \psi_{\operatorname{ext}})(T-\delta_3 + \cdot)$, and $\widehat{\xi} = \xi$, belongs to the $\varepsilon$-neighborhood of $w_0 - (\xsym{\Upsilon}(\xi, \xsym{0}) \cdot \xnab) \xi$ in $\xHn{{m-1}}$. The temperature component it not relevant, but could be brought to the $\varepsilon$-neighborhood of $\theta_0$ in $\xH^m$. Hereby, it is used that the limit in \Cref{theorem:main_smalltimelargecontrol} is uniform with respect to $(\varphi_{\operatorname{ext}}, \psi_{\operatorname{ext}})$ from bounded subsets of $\xLtwo((0,{T}); \xHn{{m-2}} \times \xHn{{m-1}})$, and that the family $((\varphi_{\operatorname{ext}}, \psi_{\operatorname{ext}})(s + \cdot))_{s \in (0,T)}$ is bounded in $\xLtwo((0,T); \xHn{{m-2}}\times\xHn{{m-1}})$ if one employs the extension $(\varphi_{\operatorname{ext}}, \psi_{\operatorname{ext}})(s) \coloneq (0,0)$ for $s \in \mathbb{R}\setminus[0,1]$. 2) By the continuity of the resolving operator for \eqref{equation:GeneralBoussCurl} from \Cref{proposition:Wellposedness}, there exists a number $\widetilde{\varepsilon} > 0$ such that the goal of the previous step is still achieved with the same value of $\delta_3$ if one replaces the initial vorticity in the application of \eqref{equation:propvortcontrl1} from the previous step by any other element of $\xHn{m}$ belonging to the $\widetilde{\varepsilon}$-neighborhood of $w_0 + \delta_3^{-1/2}\xi$ in $\xHn{{m-1}}$. Thus, via \Cref{theorem:propvortcontrl}, one can choose $0 < \delta_2 = \delta_{2,1} + \delta_{2,2} < T - \delta_3$ so small that the left-hand side in \eqref{equation:propvortcontrl0} with $\delta = \delta_{2,1}$ and $\smash{\widetilde{\xi} = -\delta_3^{-1/2}\kappa}$ belongs to the $\widetilde{\varepsilon}$-neighborhood of $w_0 + \delta_3^{-1/2}\xi$ in $\xHn{{m-1}}$. Using also \Cref{theorem:main_smalltimelargecontrol}, one could correct the temperature in time $\delta_{2,2}$ as close as desired towards $\theta_0$ in $\xHn{m}$, while keeping the vorticity in the $\widetilde{\varepsilon}$-neighborhood of $w_0 + \delta_3^{-1/2}\xi$ with respect to $\xHn{{m-1}}$. Regarding this step, when applying \Cref{theorem:propvortcontrl}, $(\varphi_{\operatorname{ext}}, \psi_{\operatorname{ext}})$ in \eqref{equation:propvortcontrl0} is replaced by $(\varphi_{\operatorname{ext}}, \psi_{\operatorname{ext}})(T-\delta_2-\delta_3 + \cdot)$.
		3) By applying \Cref{theorem:main_smalltimelargecontrol}, one can determine $0 < \delta_1 < T - \delta_2 - \delta_3$ so small that the corresponding trajectory in \Cref{theorem:main_smalltimelargecontrol} with $(\varphi_{\operatorname{ext}}, \psi_{\operatorname{ext}})$ replaced by $(\varphi_{\operatorname{ext}}, \psi_{\operatorname{ext}})(T-\delta_1-\delta_2-\delta_3 + \cdot)$ comes at $t = \delta_1$ as close to $(w_0, \theta_0 + \delta_{2,1}^{-1}\delta_3^{-1/2}\kappa)$ in $\xHn{{m-1}}\times\xHn{m}$ as desired. Thus, again by the continuity of the resolving operator for \eqref{equation:GeneralBoussCurl}, the choice of $\delta_1$ can be made so small that the goal of the previous step is still satisfied with the same value of $\delta_{2,1}$ if one takes the trajectory from the present step evaluated at $t = \delta_1$ as the initial condition.
		
		In summary, by utilizing \Cref{theorem:main_smalltimelargecontrol} and \Cref{theorem:propvortcontrl} as described above, we choose for any $\delta > 0$ the numbers $\delta_{1}, \delta_{2} = \delta_{2,1} + \delta_{2,2}, \delta_{3} \in (0, \delta/3)$, parameters $({\alpha}_{\delta_1,l})_{l \in \{1,\dots,6\}}, ({\alpha}_{\delta_{2,2},l})_{l \in \{1,\dots,6\}} \subset \xLtwo((0, 1); \mathbb{R})$, and
		\begin{equation}\label{equation:defsp}
			\begin{gathered}
				\widetilde{\gamma}_{\delta}(t) \coloneq \mathbb{I}_{(0, \delta_1)}(t) \delta_1^{-1}t + \mathbb{I}_{(\delta_1+\delta_{2,1}, \delta_1+\delta_2)}(t)\delta_{2,2}^{-1}(t-\delta_1-\delta_{2,1}), \\
				\widetilde{\gamma}_{\delta, l}(t) \coloneq \mathbb{I}_{(0, \delta_1)}(t) \delta_{1}^{-2}{\alpha}_{\delta_1,l}(\delta_1^{-1}t) + \mathbb{I}_{(\delta_1+\delta_{2,1}, \delta_1+\delta_2)}(t) \delta_{2,2}^{-2}{\alpha}_{\delta_{2,2},l}(\delta_{2,2}^{-1}(t-\delta_1-\delta_{2,1})), \\ \widetilde{\xsym{\aleph}}_{\delta}(t) \coloneq \mathbb{I}_{(0, \delta_1)} \overline{\xvec{y}}_{\delta_1}(t) + \mathbb{I}_{(\delta_1+\delta_{2,1}, \delta_1+\delta_2)}(t)\overline{\xvec{y}}_{\delta_{2,2}}(t-\delta_1-\delta_{2,1}) 
			\end{gathered}
		\end{equation}
		such that \eqref{equation:corcon} holds. From \eqref{equation:ctrlstrcaux} and \eqref{equation:defsp}, one has $(x,t) \mapsto \zeta_l(\xvec{x}, \widetilde{\gamma}_{\delta}(t)) \in \xLtwo((0,T); \xCinfty(\mathbb{T}^2;\mathbb{R}))$.
	\end{proof}

	The next theorem provides the combined global approximate controllability of the vorticity and the temperature in any given time. 
	
	\begin{thrm}\label{theorem:control_vorttemp} 
		Assume that $m \geq 2$,  $T > 0$,~$\varepsilon > 0$, $(w_0, \theta_0), (w_T, \theta_T) \in \xHn{{m-1}} \times \xHn{{m}}$, and~$(\varphi_{\operatorname{ext}}, \psi_{\operatorname{ext}}) \in \xLtwo((0,T); \xHn{{m-1}} \times \xHn{{m}})$. There exist $\gamma, \gamma_1, \dots, \gamma_{6} \in \xLtwo((0, T); \mathbb{R})$ such that the corresponding solution $(w, \xvec{u}, \theta)$ to \eqref{equation:BoussinesqVorticity}, with control $\eta$ as in \eqref{equation:degeneratecontrol}, has the regularity $(w, \theta) \in \mathcal{X}_{T}^m$ with $\smallint_{\mathbb{T}^2} \xvec{u}(\xvec{x}, \cdot) \, \xdx{\xvec{x}} \in \xCinfty_0((0,T);\mathbb{R}^2)$ and satisfies
		\begin{equation}\label{equation:rcond_vort_temp_thrm2}
			\|w(\cdot, T) - w_T\|_{m-1} + \|\theta(\cdot, T) - \theta_T\|_{m} < \frac{\varepsilon}{C_0}
		\end{equation}
		for the constant $C_0 > 0$ from \eqref{equation:velvortest}.
	\end{thrm}
	\begin{proof}
		In the following steps, the time interval $[0,T]$ is subdivided into several parts on which different control strategies are employed. At first, we note that standard estimates for the vorticity-temperature system~\eqref{equation:BoussinesqVorticity}, there exists a small time $\widetilde{\delta}_0 > 0$ such that the implication
		\[
		\|a - w_T\|_{m-1} + \|b - \theta_T\|_{m} < \frac{2\varepsilon}{3C_0} \quad \Longrightarrow \quad \|w^{\delta} - w_T\|_{m-1} + \|\theta^{\delta} - \theta_T\|_{m} < \frac{\varepsilon}{C_0}
		\]
		is true for all pairs
		\[
		(w^{\delta}, \theta^{\delta}) \coloneq S_{\delta}(a,b,\varphi_{\operatorname{ext}}(T-\delta+\cdot), \psi_{\operatorname{ext}}(T-\delta+\cdot), \xsym{0}) |_{t = \delta}, \quad 0 < \delta \leq \widetilde{\delta}_0,
		\]
		and where~$\widetilde{\delta}_0$ depends on $\varepsilon$,~$C_0$, $w_T$,~$\theta_T$, and~$(\varphi_{\operatorname{ext}}, \psi_{\operatorname{ext}})$. 
		
		\paragraph{Step 1. Starting with controls switched off.}  During the time interval $[0, T-\widetilde{\delta}_0]$, no controls will be applied. Let us denote the state of the uncontrolled trajectory at the time $t = T - \delta_0$ by
		\[
			(\widetilde{w}_0, \widetilde{\theta}_0) \coloneq S_{T - \delta_0}(w_0, \theta_0,\varphi_{\operatorname{ext}}, \psi_{\operatorname{ext}}, \xsym{0}) |_{t = T - \delta_0} \in \xHn{{m+1}} \times \xHn{{m+2}},
		\]
		where~$0 < \delta_0 \leq \widetilde{\delta}_0$ is arbitrarily selected outside a temporal set of zero measure on which the desired regularity $(\widetilde{w}_0, \widetilde{\theta}_0) \in \xHn{{m+1}} \times \xHn{{m+2}}$ is unknown. This is possible due to the assumption $(\varphi_{\operatorname{ext}}, \psi_{\operatorname{ext}}) \in \xLtwo((0,T); \xHn{{m-1}} \times \xHn{{m}})$, as it follows then from the parabolic smoothing effects of \eqref{equation:BoussinesqVorticity} (\cf~\cite{Temam2001}) that 
		\[
			S_{T}(w_0, \theta_0,\varphi_{\operatorname{ext}}, \psi_{\operatorname{ext}}, \xsym{0}) |_{(a, T)} \in \xLtwo((a, T); \xHn{{m+1}} \times \xHn{{m+2}})
		\]
		for arbitrarily fixed $0 < a < T$.
		
		\paragraph{Step 2. Warming up to special intermediate temperature profiles.}  
		Owing to \Cref{lemma:trigo}, one can choose an integer $L \geq 0$ and profiles~$\xi_0, \xi_1, \dots, \xi_{2L} \in \xCinfty(\mathbb{T}^2; \mathbb{R})$ with zero averages such that
		\begin{equation}\label{equation:epoversi1}
			\| V - w_T \|_{m-1} < \frac{\varepsilon}{6C_0}, \quad V \coloneq \widetilde{w}_0 - \partial_1\xi_0 - \sum_{i=1}^{2L}(\xsym{\Upsilon}(\xi_i, \xsym{0}) \cdot \xnab) \xi_i,
		\end{equation}
		where $\xi_i = \partial_1 \kappa_i$ for some $\kappa_i \in \xCinfty(\mathbb{T}^2; \mathbb{R})$ and $i \in \{1,\dots, 2L\}$. Now, starting at the time $t = T-\delta_0$ from the state $(\widetilde{w}_0, \widetilde{\theta}_0)$, one can use $2L$ iterated applications of \Cref{Corollary:st} to determine a finitely decomposable control for which the vorticity component of the associated controlled trajectory to \eqref{equation:GeneralBoussCurl} approximately reaches in $\xHn{{m-1}$} (in sufficiently small time) a vorticity state of the form
		\[
			\widetilde{w}_0 - \sum_{i=1}^{2L}(\xsym{\Upsilon}(\xi_i, \xsym{0}) \cdot \xnab) \xi_i + \mbox{\enquote{$2L$ large residuals}}.
		\]
		These residuals arise from the term \enquote{$\delta^{-1/2}\xi = \delta^{-1/2}\partial_1 \kappa$} in \eqref{equation:corcon}. In particular, this strategy gives rise to small numbers $\widetilde{\delta}_1, \dots, \widetilde{\delta}_{2L}$ with $\sum_{i=1}^{2L} \widetilde{\delta}_i < \delta_0/2$ such that the above mentioned control is on each time interval
		\[
			\left[T-\delta_0+\sum_{i=1}^{l} \widetilde{\delta}_i, T-\delta_0 + \sum_{i=1}^{l+1} \widetilde{\delta}_i\right], \quad l \in \{0,\dots, 2L-1\},
		\]
		defined by means of the description in \eqref{equation:defsp} for the finitely decomposable control provided by \Cref{Corollary:st}.
		Then, resorting once more to \Cref{theorem:main_smalltimelargecontrol}, a vorticity state of the form
		\begin{equation}\label{equation:form}
			\widetilde{w}_0 - \partial_1 \xi_0 - \sum_{i=1}^{2L}(\xsym{\Upsilon}(\xi_i, \xsym{0}) \cdot \xnab) \xi_i + \mbox{\enquote{$2L$ large residuals}}
		\end{equation}
		can be approximately reached in $\xHn{{m-1}}$ (in sufficiently small time) by the vorticity component of a trajectory to \eqref{equation:GeneralBoussCurl} with finitely decomposable control. Finally, to approximately reach $V$ in $\xHn{{m-1}}$ with a controlled trajectory that starts with the initial vorticity given in \eqref{equation:form}, one has to approximately cancel the aforementioned $2L$ residuals in the $\xHn{{m-1}}$-norm (in sufficiently small time). To this end, as these residuals are of the type $\widehat{\delta}^{-1/2}\partial_1 \kappa$ for different choices of $\widehat{\delta} > 0$, one can employ the limit \eqref{equation:propvortcontrl0} established in \Cref{theorem:propvortcontrl}. 
		
		Summarizing, in a manner similar to the derivation of \eqref{equation:defsp} for a single combination of Theorems~\Rref{theorem:main_smalltimelargecontrol} and~\Rref{theorem:propvortcontrl}, one obtains via the above-described steps a number $\delta_1 < \delta_0/2$, parameters $(\widetilde{\gamma}, \widetilde{\gamma}_{1}, \dots, \widetilde{\gamma}_{6})_{\delta_1 > 0} \in \xLtwo((0, T); \mathbb{R})$, and $\widetilde{\xsym{\aleph}}\in\xCinfty_0((0,T);\mathbb{R})$ such that
		\begin{equation*}
				(\widetilde{w}, \widetilde{\theta}) \coloneq S\left(\widetilde{w}_0, \widetilde{\theta}_0, \varphi_{\operatorname{ext}}(T - \delta_0+\cdot), \psi_{\operatorname{ext}}(T - \delta_0+\cdot) + \widetilde{\eta}, \widetilde{\xsym{\aleph}}\right) |_{t=\delta_1}
		\end{equation*}
		satisfies
		\begin{equation}\label{equation:epoversi2}
			(\widehat{w}_0,\widehat{\theta}_0) \coloneq (\widetilde{w}, \widetilde{\theta})(\cdot, \delta_1)  \in \xHn{{m}} \times \xHn{{m+1}}, \quad \| \widetilde{w}(\cdot, \delta_1) - V \|_{m-1} < \frac{\varepsilon}{6C_0},
		\end{equation}
		where
		\[
			\widetilde{\eta}(\xvec{x}, t) \coloneq \sum_{l = 1}^{6} 	\widetilde{\gamma}_{l}(t) \zeta_l(\xvec{x}, \widetilde{\gamma}(t)).
		\]
		The regularity condition in \eqref{equation:epoversi2} holds due to $(\varphi_{\operatorname{ext}}, \psi_{\operatorname{ext}}) \in \xLtwo((0,T); \xHn{{m-1}} \times \xHn{{m}})$ and the parabolic smoothing effects.

		\paragraph{Step 3. Cooling down to the final target temperature.} Another application of \Cref{theorem:main_smalltimelargecontrol}, with any target temperature $\widehat{\theta}_T \in \xHn{{m+1}}$ satisfying $\|\widehat{\theta}_T - \theta_T\|_m < \varepsilon/6C_0$, provides $0 < \delta_2 < \delta_0/2$ and control parameters
		\[
			(\beta_{1}, \dots, \beta_{6}) \subset \xLtwo((0,1); \mathbb{R})
		\]
		for which the trajectory
		\[
			(\widehat{w}, \widehat{\theta}) \! \coloneq \! S_{\delta_2}(\widehat{w}_0, \widehat{\theta}_0, \varphi_{\operatorname{ext}}(T-\delta_0+\delta_1+\cdot), \psi_{\operatorname{ext}}(T-\delta_0+\delta_1+\cdot) + \widehat{\eta}, \overline{\xvec{y}}_{\delta_2})
		\]
		driven by $\widehat{\eta}(\cdot, t) \coloneq  \delta_2^{-1}\sum_{l = 1}^{6} \beta_{l}(\delta_2^{-1}t) \zeta_l(\cdot, \delta_2^{-1} t)$ for $0 \leq t \leq \delta_2$ satisfies
		\[
			\|\widehat{w}(\cdot, \delta_2) - \widehat{w}_0\|_{m-1} + \|\widehat{\theta}(\cdot, \delta_2) - \widehat{\theta}_T\|_{m} < \frac{\varepsilon}{3C_0}.
		\]
		In view of \eqref{equation:epoversi1} and \eqref{equation:epoversi2}, it follows that
		\[
			\|\widehat{w}(\cdot, \delta_2) - w_T\|_{m-1} + \|\widehat{\theta}(\cdot, \delta_2) - \theta_T\|_{m} < \frac{2\varepsilon}{3C_0}.
		\]
		
		\paragraph{Step 4. Summary.} By concatenating the trajectories and controls that are obtained above on different intervals of $[0,T]$, one obtains the control parameters $\gamma, \gamma_1, \dots, \gamma_{6} \in \xLtwo((0, T); \mathbb{R})$ and a respectively controlled trajectory $(w, \xvec{u}, \theta)$ having the desired properties. In particular, since $\overline{\xvec{y}}$ from \Cref{theorem:convection} is compactly supported in $(0, 1)$, it follows that $\smallint_{\mathbb{T}^2} \xvec{u}(\xvec{x}, \cdot) \, \xdx{\xvec{x}} \in \xCinfty_0((0,T);\mathbb{R}^2)$; and in fact, the first component of the velocity average vanishes: $\smallint_{\mathbb{T}^2} u_1(\xvec{x}, \cdot) \, \xdx{\xvec{x}} = 0$.
	\end{proof}

	\subsection{Conclusion of Theorems~\Rref{theorem:main} and~\Rref{theorem:secondmain}}\label{subsection:conclusion}
	Let $r \geq 2$, $T > 0$, and $\varepsilon > 0$ be as in Theorems~\Rref{theorem:main} and~\Rref{theorem:secondmain}. An application of \Cref{theorem:control_vorttemp} with $m = r$ and $(w_0, w_T) \coloneq (\xwcurl{\xvec{u}_0}, \xwcurl{\xvec{u}_T})$ provides a solution $(\xvec{u}, w, \theta)$ to \eqref{equation:BoussinesqVorticity} that satisfies~\eqref{equation:rcond_vort_temp_thrm2}, and which is driven by a control $\eta$ of the form
	\[
		\eta(\xvec{x}, t) = \sum_{l = 1}^{6} \gamma_l(t) \zeta_l(\xvec{x}, \gamma(t)).
	\]
	However, by the proof of \Cref{theorem:control_vorttemp}, the velocity $\xvec{u}(\cdot, t)$ has nonzero average for some $t \in (0,T)$, while the temperature $\theta$ obtained via~\Cref{theorem:control_vorttemp} has zero average for all $t \in (0,T)$; see also~\eqref{equation:prevy}. As a consequence, $(\xvec{u}, \theta)$ fails to satisfy the velocity-temperature formulation \eqref{equation:BoussinesqVelocity}. In order to still pass to the velocity-temperature formulation, we will now either add a velocity control, or modify~$\theta$ appropriately. Hereto, let us denote by~$\xsym{\aleph}\in\xCinfty_0((0,T);\mathbb{R})$ the function determined by $\xsym{\aleph}\xsym{e}_{\operatorname{grav}} = \smallint_{\mathbb{T}^2} \xvec{u}(\xvec{x}, \cdot) \, \xdx{\xvec{x}} \in \xCinfty_0((0,T);\mathbb{R}^2)$.
	
	\paragraph{Option 1. Using only a localized temperature control.} Let $\chi$ be the cutoff from \eqref{equation:Definition_chi} with $\operatorname{supp}(\chi) \subset \omegaup$. Instead of~$\theta$, we consider now the function $\widetilde{\theta}(\xvec{x}, t) \coloneq \theta(\xvec{x}, t) + \chi(x_2)\aleph'(t)$,
	which satisfies (\cf~\Cref{theorem:control_vorttemp})
	\[
		\widetilde{\theta}(\cdot, 0) = 	\theta(\cdot, 0), \quad \widetilde{\theta}(\cdot, T) = 	\theta(\cdot, T).
	\]
	Allowing a slight abuse of notation, we use again the symbol $\theta$ for $\widetilde{\theta}$. Then, the triple~$(\xvec{u}, w, \theta)$ satisfies~\eqref{equation:rcond_vort_temp_thrm2} and solves~\eqref{equation:BoussinesqVelocity} with a control~$\eta$ of the form
	\begin{equation*}
		\eta(\xvec{x}) = \chi(x_2)\aleph''(t) - \tau \chi''(x_2)\aleph'(t) + u_2(\xvec{x},t) \chi'(x_2)\aleph'(t) + \sum_{l = 1}^{6} \gamma_l(t) \zeta_l(\xvec{x}, \gamma(t)),
	\end{equation*}
	which satisfies~$\operatorname{supp}(\eta) \subset \omegaup$. As we made the choices $\zeta_1 = \chi$ and~$\zeta_2 = \chi'$ in the proof of~\Cref{theorem:localized} (below \eqref{equation:ctrlstrcaux}), the proof of~\Cref{theorem:secondmain} with control of the form \eqref{equation:degeneratecontrol2} is now complete. This also implies \Cref{theorem:main}, as the continuous dependence on the data of solutions to \eqref{equation:BoussinesqVelocity} allows to approximate the degenerate control from~\Cref{theorem:secondmain} by a smooth one.

	\begin{rmrk}\label{remark:nlc}
		As mentioned in \Cref{remark:sidenote}, one could also correct the temperature average by adding~$\aleph'(t)$ to the function~$\theta$ obtained from \Cref{theorem:control_vorttemp}. While this would allow to merely employ a finitely decomposable temperature control, there would appear one term that is not physically localized in~$\omegaup$. Indeed, the term $\mathbb{I}_{\omegaup}\eta(\xvec{x}, t)$ in \eqref{equation:BoussinesqVelocity} would be replaced by $\aleph''(t) + \mathbb{I}_{\omegaup}\sum_{l = 1}^{6} \gamma_l(t) \zeta_l(\xvec{x}, \gamma(t))$, where $\aleph''(t)$ is for each $t$ a constant function with respect to the space variables.
	\end{rmrk}
	\paragraph{Option 2. One-dimensional control in the second velocity component.} The triple~$(\xvec{u}, w, \theta)$ obtained via~\Cref{theorem:control_vorttemp} satisfies the controllability condition~\eqref{equation:rcond_vort_temp_thrm2} and solves
	\begin{gather*}
		\partial_t \xvec{u} - \nu \Delta \xvec{u} + \left(\xvec{u} \cdot \xnab\right) \xvec{u} + \xnab p = (\theta + \mathbb{I}_{\omegaup} \overline{\eta}) \xvec{e}_{\operatorname{grav}} + \xsym{\Phi}_{\operatorname{ext}}, \quad \xdiv{\xvec{u}} = 0, \quad
		w(\cdot, 0) = w_0,\\
		\partial_t \theta - \tau\Delta \theta + (\xvec{u} \cdot \xnab) \theta = \mathbb{I}_{\omegaup} \left(\sum_{l = 1}^{6} \gamma_l(t) \zeta_l(\xvec{x}, \gamma(t))\right) + \psi_{\operatorname{ext}}, \quad \theta(\cdot, 0) = \theta_0,
	\end{gather*}
	where $\overline{\eta}(\xvec{x}, t) \coloneq \aleph'(t) \chi(x_2)$. We emphasize again that $\chi$ from \eqref{equation:Definition_chi} is a universal profile which only depends on $\omegaup$. Recalling the choice $\zeta_1 = \chi$ made in the proof of~\Cref{theorem:localized} (below \eqref{equation:ctrlstrcaux}), we have~$\overline{\eta}(\xvec{x}, t) = \overline{\gamma}(t)\zeta_1(x_2)$, where $\overline{\gamma}(t) \coloneq \aleph'(t)$.
	This completes the proof of \Cref{theorem:secondmain}. 
	\paragraph{Conclusion.}
	Since $\xvec{x} \mapsto \chi(x_2)\xsym{e}_{\operatorname{grav}}$ is curl-free, all above-listed options for the controls ensure together with the relations in \eqref{equation:velvortest} and \eqref{equation:rcond_vort_temp_thrm2} the target condition
	\[
		\|\xvec{u}(\cdot, T) - \xvec{u}_T\|_{r} + \|\theta(\cdot, T) - \theta_T\|_{r} \leq C_0 \|w(\cdot, T) - w_T\|_{r-1} + \|\theta(\cdot, T) - \theta_T\|_{r} < \varepsilon.
	\]
	
	\paragraph{Acknowledgment.}
	We would like to thank the anonymous reviewers for their constructive remarks and suggestions.
	
	\paragraph{Data availability statement.}
	We do not analyze or generate any datasets, because our work proceeds within a theoretical and mathematical approach.
	
	\paragraph{Version of Record.}
	This version of the article has been accepted for publication but is
	not the Version of Record and does not reflect post-acceptance improvements, or any corrections. The Version of Record is available online at: \url{https://doi.org/10.1007/s00205-025-02128-6}. Use of this Accepted Version is subject to the publisher’s Accepted Manuscript terms of use 
	
	\noindent\allowdisplaybreaks\url{https://www.springernature.com/gp/open-research/policies/accepted-manuscript-terms}
	
	\addcontentsline{toc}{section}{References}
	
	\bibliographystyle{alpha}
	\bibliography{BibBoussinesq}
  	
\end{document}